\documentclass[12pt]{amsart}

\usepackage{psfrag}
\usepackage{color}

\usepackage{graphicx,graphics}
\usepackage{fullpage,amssymb,amsfonts,amsmath,amstext,amsthm,amscd,enumerate,verbatim,tikz}
\usepackage[T1]{fontenc}
\usetikzlibrary{matrix,arrows}

\usepackage[all]{xy} \usepackage{url}

\begin{document}
\newcommand{\note}[1]{\marginpar{\tiny #1}}
\newtheorem{theorem}{Theorem}[section]
\newtheorem{result}[theorem]{Result}
\newtheorem{fact}[theorem]{Fact}
\newtheorem{conjecture}[theorem]{Conjecture}
\newtheorem{lemma}[theorem]{Lemma}
\newtheorem{proposition}[theorem]{Proposition}
\newtheorem{corollary}[theorem]{Corollary}
\newtheorem{facts}[theorem]{Facts}
\newtheorem{question}[theorem]{Question}
\newtheorem{props}[theorem]{Properties}
\theoremstyle{definition}
\newtheorem{example}[theorem]{Example}
\newtheorem{definition}[theorem]{Definition}
\newtheorem{remark}[theorem]{Remark}

\newcommand{\notes} {\noindent \textbf{Notes.  }}
\renewcommand{\note} {\noindent \textbf{Note.  }}
\newcommand{\defn} {\noindent \textbf{Definition.  }}
\newcommand{\defns} {\noindent \textbf{Definitions.  }}
\newcommand{\x}{{\bf x}}
\newcommand{\z}{{\bf z}}
\newcommand{\B}{{\bf b}}
\newcommand{\V}{{\bf v}}
\newcommand{\T}{\mathcal{T}}
\newcommand{\Z}{\mathbb{Z}}
\newcommand{\Hp}{\mathbb{H}}
\newcommand{\D}{\mathbb{D}}
\newcommand{\R}{\mathbb{R}}
\newcommand{\N}{\mathbb{N}}
\renewcommand{\B}{\mathbb{B}}
\newcommand{\C}{\mathbb{C}}
\newcommand{\dt}{{\mathrm{det }\;}}
 \newcommand{\adj}{{\mathrm{adj}\;}}
 \newcommand{\0}{{\bf O}}
 \newcommand{\w}{\omega}
 \newcommand{\av}{\arrowvert}
 \newcommand{\zbar}{\overline{z}}
 \newcommand{\htt}{\widetilde{h}}
\newcommand{\ty}{\mathcal{T}}
\renewcommand\Re{\operatorname{Re}}
\renewcommand\Im{\operatorname{Im}}
\newcommand{\diam}{\operatorname{diam}}
\newcommand{\dist}{\text{dist}}
\newcommand{\ds}{\displaystyle}
\numberwithin{equation}{section}
\newcommand{\cN}{\mathcal{N}}
\renewcommand{\theenumi}{(\roman{enumi})}
\renewcommand{\labelenumi}{\theenumi}
\newcommand{\inte}{\operatorname{int}}

\newcommand{\dn}[1]{{\scriptsize \color{blue}\textbf{Dan's note:} #1 \color{black}\normalsize}}
\newcommand{\af}[1]{{\scriptsize \color{red}\textbf{Alastair's note:} #1 \color{black}\normalsize}}

\date{\today}
\title{Normal families and quasiregular mappings}
\author{Alastair N. Fletcher}
\address{Department of Mathematical Sciences, Northern Illinois University, DeKalb, IL 60115-2888, USA\\
\textsc{\newline \indent {\includegraphics[width=1em,height=1em]{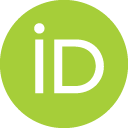} {\normalfont https://orcid.org/0000-0003-1942-6928}}}}
\email{fletcher@math.niu.edu}
\thanks{The first author was supported by a grant from the Simons Foundation, \#352034}
\author{Daniel A. Nicks}
\address{School of Mathematical Sciences, The University of Nottingham, University Park, Nottingham, NG7 2RD, UK\\
\textsc{\newline \indent {\includegraphics[width=1em,height=1em]{orcid2.png} {\normalfont https://orcid.org/0000-0002-9493-2970}}}}
\email{dan.nicks@nottingham.ac.uk}

\maketitle

\begin{abstract}
Beardon and Minda gave a characterization of normal families of holomorphic and meromorphic functions in terms of a locally uniform Lipschitz condition. Here, we generalize this viewpoint to families of mappings in higher dimensions that are locally uniformly continuous with respect to a given modulus of continuity. Our main application is to the normality of families of quasiregular mappings through a locally uniform H\"older condition. This provides a unified framework in which to consider families of quasiregular mappings, both recovering known results of Miniowitz, Vuorinen and others, and yielding new results. In particular, normal quasimeromorphic mappings, Yosida quasiregular mappings and Bloch quasiregular mappings can be viewed as classes of quasiregular mappings which arise through consideration of various metric spaces for the domain and range. We give several characterizations of these classes and obtain upper bounds on the rate of growth in each class.
\end{abstract}

\section{Introduction}

\subsection{Background}

Normal family theory plays a pivotal role in classical complex analysis and in the more modern developments of complex dynamics.
However, checking whether or not a given family is normal directly from the definition is usually a challenging task. For this reason, results giving conditions which imply normality are highly desirable. As an example, a family of meromorphic functions defined on a domain $U\subset \C$ that all omit the same three points in $\overline{\C}$ is normal by Montel's Theorem. This result has huge implications for complex dynamics, where the stable or chaotic behaviour of the iterates of a function is characterized through a condition involving normal families.

Normal family theory itself is well developed, see for example the monograph of Schiff \cite{Schiff}. Beardon and Minda \cite{BM} brought a new viewpoint to the theory by considering families of locally uniformly Lipschitz functions. They gave a characterization of normality in terms of the local Lipschitz behaviour and were then able to recover many familiar results from normal family theory and obtain some generalizations. A key to the generality of their method is that the domain and range under consideration are considered as metric spaces, and then the local uniform Lipschitz condition is given in terms of these distance functions. Changing the metrics in domain and range yields different results. It is also worth observing that their results contain normal family theory with respect to holomorphic and meromorphic functions as a special case. For example, their framework includes families of mappings which do not necessarily preserve orientation.

In the final paragraph of their paper, Beardon and Minda suggest that there should be a generalization of their approach to the setting of quasiregular mappings in two or more dimensions. The purpose of this paper is to address this question.

Quasiregular mappings provide the natural setting for generalizing geometric function theory in the plane to $\R^n$, for $n\geq 3$. Liouville's Theorem asserts that the only conformal mappings in $\R^n$, for $n\geq 3$, are M\"obius mappings and hence it is necessary to consider mappings with some distortion to get a rich theory. Fortunately, there are many analogues of important results in complex analysis for quasiregular mappings. Of particular relevance to the current paper is a generalization of Montel's Theorem to families of $K$-quasiregular mappings. We refer to Rickman's monograph \cite{Rickman} for a complete introduction to quasiregular mappings.

The Beardon-Minda viewpoint applies to families of locally uniformly Lipschitz mappings. However, $K$-quasiregular mappings on $\R^n$ need not be locally Lipschitz and are only guaranteed to be locally H\"older continuous with exponent $\alpha = K^{1/(1-n)}$, as shown by examples of the form $f_{t}(x) = x|x|^{t-1}$, for $t>0$. It is therefore natural to consider families of locally uniformly H\"older continuous mappings with a fixed H\"older exponent. 

A key role is played by the domains and ranges for our families of mappings. We will always assume that our domains and ranges are subsets of the $n$-sphere $S^n$, for $n\geq 2$, equipped with a conformal metric. Important examples to bear in mind are the unit ball $\B^n$ equipped with the hyperbolic metric, $\R^n$ equipped with the usual Euclidean metric and $S^n$ equipped with the spherical metric. Some of our work also applies to more general conformal metrics, and it is worth pointing out that there is a great deal of interplay between quasiregular mappings and metric spaces, see for example the recent book of Hariri, Kl\'en and Vuorinen \cite{HKV}.

There are several families of holomorphic and meromorphic functions in the plane which have been well-studied and can be placed into the Beardon-Minda framework. Lehto and Virtanen \cite{LV} introduced normal meromorphic functions. These are meromorphic functions $f:\D \to \overline{\C}$ for which $\{f \circ A : A \in G \}$ is a normal family, where $G$ is the automorphism group of $\D$. Bloch functions were first systematically studied by Pommerenke \cite{Po}. These are holomorphic functions $f:\D \to \C$ for which $\{ f(A(z)) - f(A(0)) : A\in G \}$ is normal, where again $G$ is the automorphism group of $\D$. Finally, Yosida \cite{Yosida} introduced the class that would become eponymous via the meromorphic functions $f: \C \to \overline{\C}$ for which $\{ f(z+w) : w\in \C \}$ is normal. Each of these three classes can be classified via a uniform continuity condition in terms of the appropriate choices of the hyperbolic metric on $\D$, the Euclidean metric on $\C$ or the spherical metric on $\overline{\C}$ in domain and range. 
There has been much research on these classes since their introduction. The interested reader may consult the books of Pommerenke \cite{Pommerenke, Pom2} and Schiff \cite{Schiff} for further developments in two dimensions.

It is natural to ask if there are corresponding classes of quasiregular mappings. In fact, the quasiregular analogue of normal meromorphic functions from $\B^n$ to $S^n$ have been well studied, see for example Vuorinen \cite{Vuorinen}, as have the analogue of Yosida functions (see \cite{BH,MV}). Via our generalization of the Beardon-Minda framework, we provide a unified viewpoint for these two classes of mappings, as well as the quasiregular analogue of Bloch functions. This latter class of mappings is studied here systematically for the first time, although see also \cite{Eremenko,Rajala}.

\subsection{Outline of the paper}

In section 2, we recall the notion of $\omega$-continuity and consider general families $\mathcal{F}$ of mappings which satisfy a local uniform $\omega$-continuity condition between subsets $X,Y$ of $S^n$ with conformal distance functions $d_X,d_Y$ respectively. We recall how $C(X,Y)$ can be made into a metric space and then interpret the Arzela-Ascoli Theorem for families of locally uniformly $\w$-continuous mappings in Theorem \ref{thm:1}. This states that, with appropriate conditions on $X$ and $Y$, a family $\mathcal{F}$ of locally uniformly $\w$-continuous mappings is relatively compact in $C(X,Y)$ if and only if there is some point $x_0 \in X$ with a relatively compact orbit $\mathcal{F}(x_0)$ in $Y$.

We further study the situation where $Y$ is a subset of another metric space $(Z,d_Z)$. Indeed, this will always be the case, since $Y\subset S^n$. The Escher condition, introduced by Beardon and Minda, holds if balls of a fixed radius in $Y$ near $\partial Y$ are contained in small balls with respect to $d_Z$. Theorem \ref{thm:escnormal} then asserts that in this situation, with appropriate conditions on $X$, $Y$ and $Z$, a family $\mathcal{F} \subset C(X,Y)$ of locally uniformly $\w$-continuous mappings is relatively compact in $C(X,Z)$ if and only if  $\mathcal{F}$ is a normal family relative to $Z$. Here, a family $\mathcal{F} \subset C(X,Y)$ is said to be normal relative to $Z$ if 
$\mathcal{F}$ is relatively compact in $C(X,Z)$ and the closure of $\mathcal{F}$ in $C(X,Z)$ is the closure of $\mathcal{F}$ in $C(X,Y)$ together possibly with constant maps into $\partial Y$.  At this point we should bear in mind the usual definition of a normal family of holomorphic functions in the plane: any sequence must have a subsequence which either converges locally uniformly, or diverges locally uniformly to infinity. This last part is the interpretation of {\it constant maps into $\partial Y$} when $Y$ is $\C$ with the Euclidean metric and $\partial Y$ is the point at infinity in $\overline{\C}$.

In section 3, we specialize to consider families of quasiregular mappings. We recall a selection of facts from this well-developed theory that will be salient for our purposes, most importantly Theorem \ref{thm:ric} that states quasiregular mappings are locally H\"older continuous. We will therefore be able to apply the results of Section 2 in the class of quasiregular mappings. 
There are various ways of considering normal families in this setting.
If $X,Y$ are subdomains of $S^n$, then we can either define a normal family of quasiregular mappings relative to $S^n$ as above, or use the definition introduced by Miniowitz \cite{Min} involving local uniform convergence, or use relative compactness in $C(X,S^n)$. Definition \ref{def:qrnormal}, and the discussion after it, will show that all these viewpoints coincide in our setting.

Our main result in section 3, Theorem \ref{thm:qr}, asserts that a family of $K$-quasiregular mappings from $X$ to $Y$ is relatively compact in $C(X,Y)$ if and only if $\mathcal{F}$ satisfies a local uniform H\"older condition with exponent $K^{1/(1-n)}$ and there is some point with a relatively compact orbit in $Y$. In the special case where $Y$ is the compact set $S^n$, this second condition can be dropped. In the rest of section 3, we recover old results and obtain new ones by considering various domains and ranges in Theorem \ref{thm:qr}, as well as providing an application to Julia sets of uniformly quasiregular mappings in Theorem \ref{thm:julia}.

In section 4, we apply our analysis of families of $K$-quasiregular mappings to study normal quasiregular mappings. Here, we assume $X\subset S^n$ carries a transitive collection $G$ of isometries. A quasiregular mapping $f:X\to S^n$ is then called normal if the family $\{ f\circ A : A\in G \}$ is a normal family. 
If the range is $\R^n$, then we instead require the family $\{ f(A(x)) - f(A(x_0)) :A\in G \}$ to be normal, for some $x_0 \in X$; compare with the definition of Bloch functions above.

Interestingly, the choice of $G$ turns out not to matter in this definition. While $G$ could be chosen to be the isometry group of $X$ (if such a group if transitive), it need not be. This is seen via Theorem \ref{thm:balls}, which characterizes normal quasiregular mappings as those which are uniformly continuous. This latter condition is independent of the choice of $G$. To illustrate the point, Yosida functions in the plane are defined by letting $G$ be the translation group of~$\C$. Nothing is lost or gained by instead letting $G$ be the full isometry group of $\C$.
We further show in Theorem \ref{thm:globalholder} that normal quasiregular mappings into $S^n$ are precisely the globally H\"older quasiregular mappings. The corresponding result for normal quasiregular mappings into $\R^n$, Theorem \ref{thm:globalholder2}, shows that such maps are H\"older on small scales and Lipschitz on large scales. 

Moving into section 5, we use our framework from section 4 to study three classes of normal quasiregular mappings. First, we introduce the quasiregular analogue of Bloch functions where the domain is $\B^n$ with the hyperbolic metric and the range is $\R^n$ with the Euclidean metric. Our first main result here, Theorem \ref{thm:blocheq}, shows that $f$ is a Bloch mapping if and only if the Bloch radius is finite. The failure of quasiregular mappings to be closed under addition means there is no Banach space structure so there is no precise analogue of the Bloch norm. Instead we give a discrete version $R_f$ via the supremum of Euclidean diameters of images of hyperbolic balls of radius $1$ and prove in Theorem \ref{lem:blochnorm} that Bloch mappings can be characterized via the finiteness of $R_f$. We also give in Theorem \ref{thmLblochgrowth} an upper bound for the rate of growth of Bloch mappings in terms of $R_f$, and construct an example that is an analogue of $\log(1-z)$.

We next turn to normal quasiregular mappings $\B^n \to S^n$. This class has already been studied, see \cite[Chapter 13]{Vuorinen}. 
Our main new contribution here is in Theorem \ref{thm:normalgrowth} where it is shown that normal quasiregular mappings have order of growth at most $n-1$. Finally, we study Yosida quasiregular mappings from $\R^n$ to $S^n$. Our main contribution here is Theorem~\ref{thm:yosida}, a generalization of a result of Minda \cite{Minda}, that asserts $f$ is not a Yosida mapping if and only if there is a rescaling in the Pang-Zalcman sense which converges to a non-constant Yosida mapping.

\section{The Beardon and Minda viewpoint in higher dimensions}

{\bf Notation:} For $n\geq 2$, we denote by $\R^n$ Euclidean $n$-space, by $S^n$ the unit $n$-sphere in $\R^{n+1}$, which we will identify with $\R^n \cup \{\infty \}$, and by $\B^n$ the open unit ball in $\R^n$. Given a metric space $(X,d_X)$, the open ball centred at $x_0 \in X$ of radius $r>0$ is $B_X(x_0,r)$. If the underlying metric space is Euclidean space, we drop the $X$ subscript.

\subsection{Conformal metrics}

Suppose $X$ is a domain in $S^n $. A conformal metric on $X$ is a continuous form $\tau(x) |dx|$ for which $\tau$ is strictly positive. A conformal metric induces a distance function given by
\[ d(u,v) = \inf \int_{\gamma} \tau (x) |dx|,\]
where the infimum is over all paths in $X$ joining $u$ and $v$. We will often write $\tau_X$ and $d_X$ for the metric and distance function on $X$, as long as the context is clear. 

\begin{example}
\label{ex:1}
Suppose $n\geq 2$.
\begin{itemize}
\item The Euclidean metric on $\R^n$ arises from $\tau(x) = 1$. 
\item Modelling $S^n$ as the unit sphere embedded in $\R^{n+1}$, the spherical distance $\sigma $ on $S^n$ arises from the restriction of the Euclidean metric on $\R^{n+1}$. 
It will be useful to note that by identifying $S^n$ with $\R^n \cup \{ \infty \}$, we have
\begin{equation}
\label{eq:sigrn} 
\sigma (u,v) = \inf_{\gamma} \int_{\gamma} \frac{ 2 \: |dx| }{1+|x|^2} ,\quad u,v \in  S^n .
\end{equation}
\item The hyperbolic distance $\rho$ on $\B^n$ arises from $\tau(x) = 2(1-|x|^2)^{-1}$.
\item If $X$ is a proper subdomain of $\R^n$, then the quasihyperbolic distance arises from $\tau(x) = d(x,\partial X)^{-1}$, where $d$ here denotes the Euclidean distance.
\end{itemize}
\end{example}

We remark that not every distance function arises in this way.

An important point to consider is whether the metrics are complete, that is, whether Cauchy sequences in the corresponding distance function converge. For example, the hyperbolic metric is complete on $\B^n$, whereas the Euclidean metric is not. However, every proper subdomain of $S^n$ carries a complete conformal metric, namely the quasihyperbolic metric. It is sometimes convenient to assume that metrics are complete, and then appeal to the following lemma.

\begin{lemma}
\label{lemma:bilip}
If $\tau_j$ are conformal metrics on $X_j$ for $j=1,2$ with $X_1 \subset X_2$, then the associated distance functions $d_1,d_2$ are bi-Lipschitz equivalent on compact subsets of $X_1$.
\end{lemma}

The proof of this lemma is the same as that of \cite[Lemma 2.1]{BM}.
We recall that $d_1$ and $d_2$ are {\it bi-Lipschitz equivalent} distance functions on $X$ if there exists $L\geq 1$ such that
\[ \frac{d_2(x,y)}{L} \leq d_1 (x,y) \leq Ld_2(x,y)\]
for all $x,y\in X$. The smallest such $L$ for which the above inequalities hold is called the {\it isometric distortion}.

\subsection{$\omega$-continuity}

The notion of $\omega$-continuity is a quantitative way of describing continuity and is closely related to uniform continuity. In this section, $X$ and $Y$ are subdomains of $S^n$ with distance functions $d_X$ and $d_Y$, respectively, arising from conformal metrics.

We recall that $C(X,Y)$ is the space of continuous maps $f:X\to Y$, where continuity is with respect to $d_X$ and $d_Y$. The space $C(X,Y)$ carries a distance function defined as follows (c.f. \cite[\S 3]{BM}): let $(K_k)_{k=1}^{\infty}$ be a compact exhaustion of $X$ and for $f,g \in C(X,Y)$, set
\[ d_k(f,g) = \sup \{ d_Y( f(x) , g(x) ) : x\in K_k \} \]
and
\[ d_{X,Y}(f,g) = \sum_{k=1}^{\infty} \frac{1}{2^k} \left ( \frac{d_k(f,g) }{1+d_k(f,g)} \right ).\]
Then $d_{X,Y}(f_m,f) \to 0$ if and only if $f_m\to f$ uniformly on compact subsets of $X$. The topology arising from this distance function is called the topology of uniform convergence on compact sets, or the topology of local uniform convergence, and does not depend on the particular choice of compact exhaustion. Consequently, when we say $f_m\to f$ in $C(X,Y)$, this means local uniform convergence relative to $d_X$ and $d_Y$ or, equivalently, convergence relative to the distance $d_{X,Y}$ on $C(X,Y)$. Note also that the topology induced on $C(X,Y)$ is not changed when the distance functions $d_X$ and $d_Y$ are replaced with topologically equivalent distances.

We will often want to talk about a family $\mathcal{F}$ of functions in $C(X,Y)$ being {\it relatively compact}. This means that the closure of $\mathcal{F}$ in $C(X,Y)$ is compact.
Since sequential compactness coincides with compactness in metric spaces, we are free to use either notion when discussing compact sets in $C(X,Y)$. In particular, $\mathcal{F}$ is relatively compact if and only if any sequence of functions $f_m \in \mathcal{F}$ is guaranteed to a have a subsequence which converges uniformly on compact subsets to an element of $C(X,Y)$.

We now recall the notion of $\omega$-continuity.

\begin{definition} \label{def:w-cts}
Let $\w:[0,\infty) \to [0,\infty)$ be a continuous increasing function with $w(0) = 0$. We call a map $f:X\to Y$ $\w$-continuous if $f$ has modulus of continuity $\w$, that is, if
\[ d_Y (f(x),f(y)) \leq \w ( d_X(x,y))\]
for all $x,y\in X$.
\end{definition}

\begin{example}
If $\w(t) = Ct$ for $C>0$ then $f$ is Lipschitz continuous. If $\w(t) = Ct^{\alpha}$ for constants $C,\alpha >0$ then $f$ is $\alpha$-H\"older continuous.
\end{example}

\begin{definition}
A family $\mathcal{F} \subset C(X,Y)$ is called:
\begin{enumerate}[(i)]
\item uniformly $\w$-continuous if there exists $L>0$ such that for all $x,y \in X$ and all $f\in \mathcal{F}$, 
\[ d_Y(f(x),f(y)) \leq L \w(d_X(x,y));\]
\item uniformly $\w$-continuous on compact sets if for each compact set $E\subset X$, there exists $L>0$ such that for all $x,y\in E$ and all $f\in \mathcal{F}$,
\[ d_Y(f(x),f(y)) \leq L \w(d_X(x,y));\]
\item locally uniformly $\w$-continuous if for each $x_0\in X$, there exists $r,L>0$ such that for all $x,y \in B_X(x_0,r)$ and all $f\in \mathcal{F}$,
\[ d_Y (f(x),f(y)) \leq L \w (d_X(x,y)).\]
\end{enumerate}
\end{definition}

We will show that uniform $\w$-continuity on compact subsets and local uniform $\w$-continuity are the same notions with respect to conformal metrics.

\begin{proposition}
\label{prop:1}
Let $X,Y$ be subdomains of $S^n$ equipped with conformal metrics and corresponding distance functions $d_X,d_Y$, respectively. 
A family $\mathcal{F} \subset C(X,Y)$ is locally uniformly $\w$-continuous if and only if it is uniformly $\w$-continuous on compact sets.
\end{proposition}

\begin{proof}
Each point in $X$ has a compact neighbourhood, so uniformly $\w$-continuous on compact subsets implies locally uniformly $\w$-continuous.

For the converse, suppose that $\mathcal{F}$ is locally uniformly $\w$-continuous but not uniformly $\w$-continuous on compact sets.  Then there exist a compact set $E\subset X$, points $x_m, y_m \in E$ and $f_m\in\mathcal{F}$ such that, for all $m\in\N$,
\begin{equation}
\label{eq:d(fn(xn),fn(yn))} 
 d_Y (f_m(x_m),f_m(y_m)) > m \w (d_X(x_m,y_m) ).
\end{equation}
We may assume that the sequences $(x_m)$ and $(y_m)$ converge to limits $x$ and $y$ in $E$ respectively. We must have that $x\ne y$, otherwise we obtain a contradiction from the local uniform $\w$-continuity of $\mathcal{F}$ at $x$. Using local uniform $\w$-continuity at $x$ and $y$, it now follows from \eqref{eq:d(fn(xn),fn(yn))} that
\begin{equation}
\label{eq:d_Y to infty} 
 d_Y (f_m(x),f_m(y)) \to \infty, \quad \mbox{ as } m\to\infty.
\end{equation}

Let $\gamma: [0,1] \to X$ be a path from $x$ to $y$. For each point $u\in\gamma([0,1])$ there exist $0<r_u<\frac12$ and $L_u>0$ such that for all $p,q\in B_X(u,r_u)$ and all $f\in\mathcal{F}$,
\[ d_Y (f(p),f(q)) \leq L_u \w (d_X(p,q) ) \le L_u\w(1). \]
Since $\gamma([0,1])$ is compact, we can cover it with finitely many such balls, say $B_X(u_i,r_{u_i})$ for $i=1,\ldots J$. It follows that, for all $f\in\mathcal{F}$,
\[ d_Y(f(x), f(y)) \le \sum_{i=1}^J L_{u_i}\w(1) \]
in contradiction to \eqref{eq:d_Y to infty}.
\end{proof}

\subsection{Relative compactness and the Arzela-Ascoli Theorem}

We recall the definition of equicontinuity.

\begin{definition}
\label{def:equi}
Let $X,Y$ be subdomains of $S^n$ with conformal metrics and associated distance functions $d_X,d_Y$ respectively. Let $\mathcal{F} \subset C(X,Y)$.
\begin{enumerate}[(i)]
\item $\mathcal{F}$ is {\it equicontinuous at $x_0 \in X$} if  for every $\epsilon >0$, there exists $\delta >0$ such that $d_Y(f(x),f(x_0)) < \epsilon$ whenever $d_X(x,x_0) < \delta$ and $f\in \mathcal{F}$.
\item $\mathcal{F}$ is {\it equicontinuous on $X$} if  it is equicontinuous at each point $x_0 \in X$.
\end{enumerate}
\end{definition}

The Arzela-Ascoli Theorem in this setting takes the form (\cite[Theorem~47.1]{Munkres}):

\begin{theorem}
\label{thm:aa}
Let $\mathcal{F}$ be a family of continuous functions from a locally compact Hausdorff metric space $X$ to a metric space $Y$. Then 
$\mathcal{F}$ is relatively compact in $C(X,Y)$ if and only if
\begin{enumerate}[(i)]
\item the family $\mathcal{F}$ is equicontinuous on $X$ and
\item for every $x\in X$, the orbit $\mathcal{F}(x) = \{f(x) : f\in\mathcal{F} \}$ is relatively compact in $Y$.
\end{enumerate}
\end{theorem}

If $X,Y$ are subdomains of $S^n$ with conformal metrics, then they satisfy the hypotheses of the theorem.
We can use the Arzela-Ascoli Theorem to formulate a characterization of relative compactness in the setting of $\w$-continuity, c.f. \cite[Theorem 7.1]{BM}, where the completeness of $d_Y$ shows that only one orbit needs to be checked for relative compactness. 

\begin{theorem}
\label{thm:1}
Let $X,Y$ be subdomains of $S^n$ with conformal metrics and associated distance functions $d_X,d_Y$ respectively and suppose the metric on $Y$ is complete. Let $\mathcal{F} \subset C(X,Y)$ be locally uniformly $\omega$-continuous. Then
$\mathcal{F}$ is relatively compact in $C(X,Y)$ if and only if there exists $x_0\in X$ such that $\mathcal{F}(x_0) = \{ f(x_0) : f\in \mathcal{F} \}$ is relatively compact in $Y$.
\end{theorem}

\begin{proof}
If $\mathcal{F}$ is relatively compact, then the conclusion is just the Arzela-Ascoli Theorem. For the converse, the local uniform $\w$-continuity condition implies that $\mathcal{F}$ is equicontinuous on $X$. Fix $x\in X$. Then the set $\{ x,x_0\}$ is compact and $\mathcal{F}$ is uniformly $\w$-continuous on $\{x,x_0\}$ by Proposition \ref{prop:1}. The relative compactness of $\mathcal{F}(x)$ now follows from the relative compactness of $\mathcal{F}(x_0)$, using the fact that all closed balls in $(Y,d_Y)$ are compact because $d_Y$ is complete and is obtained by integrating a conformal metric. The Arzela-Ascoli Theorem then implies that $\mathcal{F}$ is relatively compact in $C(X,Y)$. 
\end{proof}

We remark that the completeness of $d_Y$ is necessary in Theorem \ref{thm:1}, as the next example shows.
\begin{example}
Let $X=Y = (-1,1)^2$ equipped with the Euclidean distance. For $m\geq 2$ and $x\in (-1,1)$ consider the piecewise linear maps
\[ p_m(x) = \left \{ \begin{array}{cc} x & x\in (-1,0], \\ \frac{2(m-1)x}{m} & x\in (0,1/2), \\ \frac{ 2x+m-2}{m}   & x\in [1/2 ,1), \end{array} \right .\]
and define $f_m(x,y) = (p_m(x) ,y )$ for $(x,y) \in X$. 
The family $\{ f_m : m\geq 2 \} \subset C(X,Y)$ is uniformly Lipschitz, but the orbit of $(0,0)$ is relatively compact in $Y$ and the orbit of $(1/2,0)$ is not.
\end{example}

\subsection{The Escher condition}

\begin{definition}
Suppose that $Y\subset Z \subset S^n$ with distance functions $d_Y,d_Z$. Following \cite[\S 10]{BM} we say that $d_Y$ satisfies an Escher condition relative to $d_Z$ if for all $R>0$ and $\epsilon >0$ there exists $E\subset Y$ compact so that if $x\in Y \setminus E$ and $d_Y(x,y) <R$ then $d_Z(x,y) \leq \epsilon d_Y(x,y)$. We may also refer to an Escher condition between the corresponding metrics.
\end{definition}

The idea is that near $\partial Y$, disks of radius $R$ in $d_Y$ are contained in small disks in $d_Z$. It is worth bearing in mind the example of the hyperbolic metric in $\B^n$: near $\partial \B^n$, hyperbolic disks of radius $1$ have small Euclidean diameter. 

\begin{theorem}
\label{thm:esccond}
Suppose that $\tau_Y$ is a complete metric on $Y$ and $Y\subset Z \subset S^n$ with metric $\tau_Z$. Then
$\tau_Y$ satisfies an Escher condition relative to $\tau_Z$ if and only if
\[ \lim_{x\to \zeta} \frac{ \tau_Z(x)}{\tau_Y(x)} = 0,\]
for all $\zeta \in \partial Y$.
\end{theorem}

The proof is identical to \cite[Theorem 10.1]{BM} and is omitted. The classical definition of a normal family of holomorphic functions in $\C$ requires the limit of a convergent subsequence to either be holomorphic or for the subsequence to diverge to infinity uniformly on compact subsets. The latter situation should be viewed in the context of the Riemann sphere and the subsequence converging uniformly on compact subsets to the point at infinity. In light of this, we recall the definition of a normal family from \cite[\S 9]{BM}.

\begin{definition}
\label{def:nfr}
Suppose $X,Y,Z$ are domains in $S^n$ equipped with conformal metrics, $Y\subset Z$ and $\mathcal{F} \subset C(X,Y)$. Then $\mathcal{F}$ is a {\it normal family relative to $Z$} if $\mathcal{F}$ is relatively compact in $C(X,Z)$ and the closure of $\mathcal{F}$ in $C(X,Z)$ is the closure of $\mathcal{F}$ in $C(X,Y)$ together possibly with constant maps into $\partial Y$, viewing the boundary of $Y$ as a subset of $Z$.
\end{definition}

We remark that the specific choice of conformal metrics on $X,Y$ and $Z$ is irrelevant in this definition, since any choices of conformal metrics generate the same topology on $C(X,Y)$ and $C(X,Z)$.
In the discussion preceding this definition, $X$ and $Y$ are $\C$ with the Euclidean metric and $Z$ is the Riemann sphere with the spherical metric. We next show how, for locally uniformly $\omega$-continuous families, the Escher condition allows us to use relative compactness and normality as synonyms, c.f. \cite[Theorem 11.1]{BM}.

\begin{theorem}
\label{thm:escnormal}
Let $X,Y,Z$ be subdomains of $S^n$ with conformal metrics and corresponding distance functions $d_X,d_Y,d_Z$, respectively, and $Y\subset Z$. Suppose that $Y$ is relatively compact in $Z$, that $d_Y$ satisfies an Escher condition relative to $d_Z$, and that $\mathcal{F}\subset C(X,Y)$ is locally uniformly $\w$-continuous. Then $\mathcal{F}$ is relatively compact in $C(X,Z)$ if and only if $\mathcal{F}$ is a normal family relative to $Z$.
\end{theorem}

\begin{proof}
The normality of $\mathcal{F}$ relative to $Z$ implies that $\mathcal{F}$ is relatively compact in $C(X,Z)$.

For the converse, suppose that $\mathcal{F}$ is relatively compact in $C(X,Z)$.
Let $f_m \in \mathcal{F}$ be a sequence and, passing to a subsequence and relabelling if necessary, we can assume that $f_m \to f$ locally uniformly in $C(X,Z)$. We need to show that either $f$ is in $C(X,Y)$ or $f$ is constant. Let us assume that $f$ is not in $C(X,Y)$, in which case there exists  $x_0 \in X$ with $f(x_0) = \lim_{m\to \infty} f_m(x_0) \in \partial Y$.
Take $x\in X$ and $\epsilon>0$. Since $\{ x_0,x\}$ is compact and $\mathcal{F}$ is locally uniformly $\w$-continuous, Proposition~\ref{prop:1} implies that
\begin{equation}
\label{eq:unifwcts} 
d_Y(f_m(x),f_m(x_0)) \leq L \w (d_X(x,x_0))
\end{equation}
for all $m\in \N$. Set $C = L\w(d_X(x,x_0))$. The Escher condition means that we can find a compact set $E\subset Y$ such that if $v\in Y\setminus E$ and $d_Y(u,v) <C$ then 
\begin{equation}
\label{eq:escherC} 
d_Z(u,v) < \frac{\epsilon}{C}d_Y(u,v).
\end{equation}
Since $E$ is a compact subset of $Y$ and $f_m(x_0)\to f(x_0)\in\partial Y$, there is $M\in \N$ such that if $m\geq M$ then $f_m(x_0)\in Y\setminus E$ and  $d_Z(f_m(x_0) , f(x_0)) < \epsilon$.
Hence by \eqref{eq:unifwcts} and \eqref{eq:escherC}, if $m\geq M$ then
\[ d_Z(f_m(x) , f(x_0)) \leq d_Z(f_m(x) , f_m(x_0)) + d_Z(f_m(x_0),f(x_0)) < 2\epsilon.\]
It follows that $f(x) = f(x_0)$ and hence $f$ is a constant map into $\partial Y$.
\end{proof}

Note that, unlike \cite[Theorem 11.1]{BM}, the above result does not assume that the metrics on the codomains $Y$ and $Z$ are complete. Theorem~\ref{thm:escnormal}  immediately has the following corollary.

\begin{corollary}
\label{cor:esc1}
Let $X\subset \R^n$ be equipped with a conformal metric and let $\mathcal{F} \subset C(X,\R^n)$ be locally uniformly $\w$-continuous with respect to $d_X$ and the Euclidean distance on $\R^n$. Then $\mathcal{F}$ is a normal family relative to $S^n$ equipped with the spherical metric. 
\end{corollary}

\begin{proof}
Since the spherical distance $\sigma$ is bounded above by twice the Euclidean distance on $\R^n$, we immediately obtain that $\mathcal{F}$ is locally uniformly $\omega$-continuous with respect to $d_X$ and~$\sigma$. The first condition in the Arzela-Ascoli Theorem is thus satisfied. The second condition is satisfied automatically since $S^n$ is compact. Hence $\mathcal{F}$ is relatively compact in $C(X,S^n)$. Since the Euclidean and spherical metrics satisfy the Escher condition, Theorem \ref{thm:escnormal} implies that $\mathcal{F}$ is a normal family relative to $S^n$.
\end{proof}

\section{Quasiregular mappings}

The setting of quasiregular mappings in $\R^n$, for $n\geq 2$, is the natural counterpart to function theory in the plane. We refer to \cite{Rickman} for a comprehensive review of the theory of quasiregular mappings but we briefly outline the definition and important properties below.

\subsection{Preliminaries}

A {\it quasiregular mapping} in a domain $U\subset \R^n$ for $n\geq 2$ is a continuous mapping in the Sobolev space $W^1_{n,loc}(U)$ where there is a uniform bound on the distortion, that is, there exists $K\geq 1$ such that
\[|f'(x)|^n \leq KJ_f(x)\]
almost everywhere in $U$, where $|f'(x)|$ is the operator norm of $f'(x)$ and $J_f(x)$ is the Jacobian determinant of $f$ at $x$. The minimal $K$ for which this inequality holds is called the {\it outer dilatation} and denoted by $K_O(f)$. As a consequence of this, there is also $K' \geq 1$ such that 
\[J_f(x) \leq K' \inf_{|h|=1}|f'(x)h|^n\]
holds almost everywhere in $U$. The minimal $K'$ for which this inequality holds is called the {\it inner dilatation} and denoted by $K_I(f)$. We then have that $K(f)= \max \{K_O(f), K_I(f) \}$ is the maximal dilatation of $f$. A $K$-quasiregular mapping is a quasiregular mapping for which $K(f) \leq K$. An injective quasiregular mapping is called quasiconformal. It turns out that non-constant quasiregular mappings are discrete and open.

To deal either with quasiregular mappings that are defined at infinity, or with quasiregular maps that have a discrete set of poles, we may pre- or post-compose by a M\"obius map $A$ (which could be chosen to be a spherical isometry) that moves the point at infinity to the origin, and then locally applying the above condition. Such mappings are sometimes called quasimeromorphic, but we will keep the nomenclature quasiregular and bear in mind the domain and range can include the point at infinity.

The main property that we are interested in here is the fact that quasiregular mappings are locally H\"older continuous, as the following result illustrates.

\begin{theorem}[\cite{Rickman}, Theorem III.1.11]
\label{thm:ric}
Let $n\geq 2$ and $K\geq 1$.
Let $U \subset \R^n$ be a bounded domain and $f:U \to \R^n$ be a bounded $K$-quasiregular mapping. Suppose that $E\subset U$ is compact, let $\delta = d(E, \partial U)$, where $d$ denotes Euclidean distance and let $r$ denote the diameter of $f(U)$. Then
\[ |f(x) - f(y)| \leq C |x-y|^{\alpha},\]
for all $x\in E$ and $y \in U$, and where $\alpha = K^{1/(1-n)}$, $C = \lambda(n)\delta ^{-\alpha}r$ with $\lambda(n)$ a constant that depends only on $n$.
\end{theorem}

Henceforth, we will denote by $\alpha$ the constant $K^{1/(1-n)}$. Below we will prove some results on the order of growth of various classes of mappings. We recall the required notions here. 
Suppose $f:U \to S^n$ is a non-constant $K$-quasiregular mapping.
The local index $i(x,f)$ is the infimum of $\sup_y \operatorname{card}f^{-1}(y)\cap V$, where $V$ runs over all neighbourhoods of $x$.
For $y\in S^n$ and a Borel set $E$ with $\overline{E}$ compact in $U$, define
\[ n(E,y) = \sum_{x\in f^{-1}(y) \cap E} i(x,f),\]
that is, $n(E,y)$ counts the number of points in $f^{-1}(y)\cap E$ with multiplicity. For $r>0$, $n(r,y)$ is defined to be $n(\overline{B(0,r)},y)$.

We define $A_{f,\sigma}(r)$ to be the average of $n(r,y)$ over $S^n$ with respect to spherical measure, that is, 
\[ A_{f,\sigma}(r) = \frac{1}{\omega_n} \int_{\R^n} \frac{ n(r,y)}{(1+|y|^2)^n } \: dm(y) = \frac{1}{\omega_n} \int_{\overline{B(0,r)}} \frac{J_f(x)}{(1 +  |f(x)|^2)^n} \: dm(x),\]
where $dm$ denotes Lebesgue measure and $\w_n$ the surface area of $S^n$. If the context is clear, we abbreviate $A_{f,\sigma}(r)$ to $A(r)$. 
We note that our formulas differ from those in Rickman's monograph \cite{Rickman} by a factor of $2^n$ as we are identifying $\R^n \cup \{ \infty \}$ with the unit sphere $S^n$ in $\R^{n+1}$, whereas Rickman uses a sphere of radius $1/2$.
We also define
\begin{equation}
\label{eq:afr} 
A(x_0,r) = \frac{1}{\omega_n} \int_{\overline{B(x_0,r)}} \frac{J_f(x)}{(1 +  |f(x)|^2)^n} \: dm(x),
\end{equation}
which has the interpretation of the normalized spherical volume of $f(\overline{B(x_0,r)})$ with multiplicity taken into account.

If $Y$ is an $(n-1)$-dimensional sphere in $\overline{\R^n}$, then $\nu(E,Y)$ denotes the average  of $n(E,y)$ over $Y$ with respect to the $(n-1)$-dimensional spherical measure. In the particular case where $E = \overline{B(0,r)}$, we write $\nu(r,s)$ for $\nu ( \overline{B(0,r)}, \partial B(0,s))$.

For quasiregular mappings on $\R^n$, the order $\mu_f$ and lower order $\lambda_f$ of $f$ are defined by
\[ \mu_f = \limsup_{r\to \infty} \frac{ \log A(r) }{\log r}, \quad \lambda_f = \liminf_{r\to \infty} \frac{\log A(r)}{\log r}.\]
If $f$ has no poles then the order and lower order can be expressed in terms of the maximum modulus function
\[ M(r,f) = \sup_{|x| = r} |f(x)| \]
via
\[ \mu_f = \limsup_{r\to \infty} (n-1) \frac{\log \log M(r,f)}{\log r}, \quad \lambda_f = \liminf_{r\to \infty} (n-1) \frac{\log \log M(r,f)}{\log r}.\]
For quasiregular mappings in the unit ball $\B^n$, analogous definitions for order and lower order hold, with $\log r$ replaced by $\log\frac{1}{1-r}$ and limits as $r\to 1^-$.

\subsection{Families of $K$-quasiregular mappings}

We now make precise the notion of a normal family of quasiregular mappings. First we need the following definition.

\begin{definition}
If $X,Y$ are subdomains of $S^n$, then denote by $\mathcal{Q}(X,Y)$ the subset of $C(X,Y)$ consisting of quasiregular mappings from $X$ to $Y$. Moreover, for $K\geq 1$, denote by $\mathcal{Q}_K(X,Y)$ the collection of $K$-quasiregular mappings from $X$ to $Y$.
\end{definition}

The reason for this definition is that if $X,Y\subset \R^n$ are domains and $f_m\in \mathcal{Q}_K(X,Y)$ is a sequence of $K$-quasiregular mappings which converges locally uniformly to a map $f\in C(X,Y)$, then $f$ itself must be in $\mathcal{Q}_K(X,Y)$ by \cite[Theorem VI.8.6]{Rickman}. Extending to include the point at infinity in domain or range can be handled via M\"obius maps. 
In the absence of such a uniform bound on the maximal dilatation of the mappings in the family, there is no longer a guarantee that any limit function is quasiregular. Consider, for example, the planar family $\{ f_K(x+iy) = Kx+iy : K>0 \}$.

\begin{definition}
\label{def:qrnormal}
Let $X,Y$ be subdomains of $S^n$ with conformal metrics, let $K\geq 1$ and let $\mathcal{F} \subset \mathcal{Q}_K(X,Y)$. Then we say that $\mathcal{F}$ is a normal family if any (and hence all) of the following equivalent statements hold:
\begin{enumerate}
\item $\mathcal{F}$ is a normal family relative to $S^n$ in the sense of Definition~\ref{def:nfr};
\item $\mathcal{F}$ is relatively compact in $C(X,S^n)$;
\item every sequence $(f_m)$ in $\mathcal{F}$ has a subsequence that converges uniformly on compact subsets of $X$, in the spherical metric, to a limit function $f\colon X\to S^n$.
\end{enumerate}
\end{definition}

Note that (iii) is the definition stated by Miniowitz in \cite{Min}. Recall that statements (i) and (ii) are independent of the choice of conformal metrics on $X, Y$ and $S^n$, and that (ii) is equivalent to (iii); see the remarks preceding Definition~\ref{def:w-cts}. To see that (i) and (ii) are equivalent, we need to show that for $\mathcal{F}\subset \mathcal{Q}_K(X,Y)$, any map in the closure of $\mathcal{F}$ in $C(X,S^n)$ is either in the closure of $\mathcal{F}$ in $C(X,Y)$ or is a constant map into $\partial Y$. This can be obtained from the following quasiregular version of Hurwitz's Theorem, which implies that a map in the closure of $\mathcal{Q}_K(X,Y)$ in $C(X,S^n)$ must be constant if its image contains any value $a\notin Y$.

\begin{theorem}[Hurwitz's Theorem, see \cite{Min} Lemma 2]
Let $f_m: U \to \R^n \setminus \{a\}$ be a sequence of $K$-quasiregular mappings that converges locally uniformly on $U$ to a $K$-quasiregular mapping $f$. Then $f$ is either constant or $f$ omits $a$ in $U$.
\end{theorem}

There are various results in the literature describing when a family of quasiregular mappings is normal. The following version of Montel's Theorem for quasiregular mappings has found many applications.

\begin{theorem}[\cite{Min}, Theorem 4]
\label{thm:montel}
Let $n\geq 2$, $K\geq 1$ and $U \subset S^n$ be a domain. Then there exists a constant $q=q(n,K)$ such that if $\mathcal{F}\subset \mathcal{Q}_K(U,S^n)$ is a family of $K$-quasiregular mappings which all omit the same $q$ distinct points $a_1,\ldots, a_q$ in $S^n$, then $\mathcal{F}$ is a normal family.
\end{theorem}

We remark that the statement of \cite[Theorem 5]{Min} is stronger, in that the omitted points are allowed to vary, but must stay at least a fixed spherical distance apart. 
We also recall Miniowitz's version of Zalcman's Lemma, see \cite[Lemma 1]{Min}.

\begin{lemma}
\label{lem:zalcman}
A family $\mathcal{F}$ of $K$-quasiregular mappings from $\B^n$ to $S^n$ is not normal at $x_0$ if and only if there exist sequences $x_m\to x_0$, $f_m \in \mathcal{F}$ and positive reals $\rho_m \to 0^+$ such that
\[ f_m(x_m + \rho_mw) \to g(w)\]
locally uniformly on compact subsets of $\R^n$, where $g$ is a non-constant quasiregular mapping $g:\R^n \to S^n$.
\end{lemma}

Here, a family $\mathcal{F}$ is said to be \emph{normal at $x_0$} if there exists a neighbourhood $U$ of $x_0$ such that the family of restrictions $\mathcal{F} |_U$ is normal. 
We remark that Miniowitz's statement does not include the requirement that $x_m \to x_0$, but this can easily be achieved, see for example \cite[19.7.3]{IM}.

The main result of this section is the following characterization of relative compactness for families of $K$-quasiregular mappings based on $\omega$-continuity.

\begin{theorem}
\label{thm:qr}
Let $X\subset S^n$ be a domain equipped with distance function $d_X$ arising from a conformal metric and let $\mathcal{F}\subset \mathcal{Q}_K(X,Y)$ be a family of $K$-quasiregular mappings defined on $X$ with image contained in $Y \subset S^n$ equipped with distance function $d_Y$ arising from a complete conformal metric. Then $\mathcal{F}$ is relatively compact in $C(X,Y)$ if and only if
\begin{enumerate}[(i)]
\item $\mathcal{F}$ is locally uniformly $\w$-continuous, with $\w(t) = t^{\alpha}$, $\alpha = K^{1/(1-n)}$ and
\item there exists $x_0 \in X$ such that $\mathcal{F}(x_0) = \{ f(x_0) : f\in \mathcal{F} \}$ is relatively compact in $Y$.
\end{enumerate}
In particular, if $Y$ is $S^n$ with the spherical metric and if $\mathcal{F}\subset \mathcal{Q}_K(X,S^n)$ is a family of $K$-quasiregular mappings, then $\mathcal{F}$ is normal if and only if $\mathcal{F}$ is locally uniformly $\w$-continuous, with $\w(t) = t^{\alpha}$, $\alpha = K^{1/(1-n)}$.
\end{theorem}

\begin{proof}
The sufficiency follows from Theorem \ref{thm:1}. 
For the other direction, by Theorem \ref{thm:1} it suffices to prove that if $\mathcal{F}$ is relatively compact in $C(X,Y)$ then $\mathcal{F}$ is locally uniformly $\w$-continuous, with $\w(t) = t^{\alpha}$. Our main tool is Theorem \ref{thm:ric}, but we need mappings on bounded Euclidean domains to be able to apply it. 

To that end, let $x_0 \in X$, and choose $s>0$ such that the closed ball $B=\overline{B_X(x_0, s)}$ is compact. Then, since $B$ is compact and $\mathcal{F}$ is relatively compact in the topology of local uniform convergence on $C(X,Y)$, it is not hard to see that any sequence in $\bigcup_{f\in \mathcal{F}} f(B)$ has a subsequence that converges in $Y$. Hence $\bigcup_{f\in \mathcal{F}} f(B)$ is relatively compact in $Y$. We conclude that $d_Y$ and the spherical distance $\sigma$ are bi-Lipschitz equivalent on $\bigcup_{f\in\mathcal{F}} f(B)$ with, say, isometric distortion $L$.


Now for $0<r_0<s/2$ set $E = \overline{B_X(x_0,r_0)}$ and $U = B_X(x_0,2r_0)$. Since $\mathcal{F}$ is relatively compact in $C(X,Y)$, the Arzela-Ascoli Theorem implies that $\mathcal{F}$ is equicontinuous on $\overline{U}$. Taking $\epsilon =\pi/2L$, we can choose $r_0$ sufficiently small that 
\[ \diam _Y f(\overline{U}) \leq \epsilon\]
for all $f\in \mathcal{F}$. Hence
\[ \diam_{\sigma} f(\overline{U}) \leq L\epsilon = \pi/2\]
for all $f\in \mathcal{F}$. We also choose $r_0$ small enough that $\diam_{\sigma} \overline{U} \leq \pi/2$.

Note that $d_X$ and $\sigma$ are bi-Lipschitz equivalent on $\overline{U}$ with, say, isometric distortion $L_1$.
Let $A_1:S^n \to S^n$ be a spherical isometry which maps $x_0$ to $0$, and let $A_f:S^n\to S^n$ be a spherical isometry which maps $f(x_0)$ to $0$. Then $A_1(\overline{U})$ and $A_f(f(\overline{U}))$ are contained in the closed unit ball $\overline{\B^n}$ for all $f\in \mathcal{F}$. Recalling \eqref{eq:sigrn}, for $u,v\in\overline{\B^n}$, the Euclidean and spherical distances are related by $|u-v|\le \sigma(u,v)\le 2|u-v|$.
For $f\in\mathcal{F}$ we then obtain, via Theorem~\ref{thm:ric} applied to $A_f\circ f \circ A_1^{-1}$, the domain $A_1(U)$ and compact set $A_1(E)$, the following chain of inequalities for $x,y\in E$:
\begin{align*}
d_Y( f(x),f(y)) &\leq L\sigma (f(x) , f(y) ) \\
&= L \sigma ( A_f(f(x)) , A_f(f(y)) ) \\
&\leq 2L| A_f(f(x)) - A_f(f(y)) | \\
&\leq 2LC|A_1(x)-A_1(y)|^{\alpha}\\
&\leq 2LC \sigma (A_1(x) , A_1(y))^{\alpha}\\
&= 2LC \sigma (x,y)^{\alpha}\\
&\leq 2LCL_1^{\alpha} d_X(x,y)^{\alpha}.
\end{align*}
By Theorem \ref{thm:ric}, $C$ here depends only on $n,K$,  the Euclidean distance $d(A_1(E) , \partial A_1(U) )$ and the Euclidean diameter of $A_f(f(U))$. Since $A_f(f(U))$ is contained in the closed unit ball, this diameter is bounded above by $2$. The rest of the terms here are independent of the choice of $f$ and hence $C$ can be chosen not to depend on $f$. 
Since $x_0$ was arbitrary, we conclude that $\mathcal{F}$ is locally uniformly $w$-continuous on $X$.

The final statement of the theorem follows since if $Y$ is $S^n$ then (ii) always holds and, in this case, normality and relative compactness coincide.
\end{proof}

This characterization of normality of families of $K$-quasiregular mappings using the spherical distance $\sigma$ in the domain and range allows us to recover the following result of Miniowitz.

\begin{theorem}[Theorem 1, \cite{Min}]
Let $\mathcal{F}\subset \mathcal{Q}_K(X,S^n)$ be a family of $K$-quasiregular mappings on a domain $X\subset S^n$ for $n\geq 2$ equipped with the spherical distance $\sigma$.
Then $\mathcal{F}$ is a normal family if and only if for each compact subset $E\subset X$ there exists $L>0$ such that
\[ \sigma (f(x),f(y)) \leq L \w ( \sigma (x,y) ),\]
for all $x,y \in E$ and $f\in \mathcal{F}$,
where $\w(t) = t^{\alpha}$ and $\alpha = K^{1/(1-n)}$.
\end{theorem}

In the case where the domain $X$ is the unit ball with the hyperbolic metric, we have the following result, which is presumably known to the experts.

\begin{theorem}
Let $\mathcal{F}\subset \mathcal{Q}_K(\B^n,S^n)$ be a family of $K$-quasiregular mappings defined on $\B^n$ for $n\geq 2$. Then $\mathcal{F}$ is a normal family if and only if for each compact subset $E\subset \B^n$ there exists $L>0$ such that
\[ \sigma (f(x),f(y)) \leq L \w ( \rho( x,y) ),\]
for all $x,y\in E$ and $f\in \mathcal{F}$,
where $\rho$ denotes the hyperbolic distance on $\B^n$ and $\w(t) = t^{\alpha}$ with $\alpha = K^{1/(1-n)}$.
\end{theorem}

In the case where the range is $\R^n$ with the Euclidean metric, by the discussions at the start of this section, a family $\mathcal{F} \subset \mathcal{Q}_K(X,\R^n)$ is normal if and only if every sequence in $\mathcal{F}$ has a subsequence which either converges to an element of $\mathcal{Q}_K(X,\R^n)$ or diverges to infinity. In the normal family literature, a family is called {\it finitely normal} if the second case is ruled out, that is, if $\mathcal{F}$ is relatively compact in $C(X,\R^n)$. Theorem \ref{thm:qr} then yields the following.

\begin{corollary}
\label{thm:qr3}
Let $X\subset S^n$ be a domain equipped with distance function $d_X$ arising from a conformal metric and let $\mathcal{F}\subset \mathcal{Q}_K(X,\R^n)$ be a family of $K$-quasiregular mappings. Then $\mathcal{F}$ is finitely normal if and only if $\mathcal{F}$ is locally uniformly $\w$-continuous, with $\w(t) = t^{\alpha}$, $\alpha = K^{1/(1-n)}$, and $\{ f(x_0) : f\in \mathcal{F} \}$ is bounded for some $x_0 \in X$.
\end{corollary}

A recent result of Hinkkanen and Martin \cite[Theorem 1]{HM} states that if $p>0$, $\Omega$ is a plane domain and $\mathcal{F}$ is a family of $K$-quasiregular mappings $f:\Omega \to \C$ which is uniformly bounded in $L^p(\Omega)$, then $\mathcal{F}$ is finitely normal. Combining this with Corollary \ref{thm:qr3}, we obtain the following.

\begin{corollary}
\label{cor:hm}
Let $\mathcal{F}$ be a family of $K$-quasiregular mappings without poles on a plane domain which is uniformly bounded in $L^p$, for some $p>0$. Then $\mathcal{F}$ is locally uniformly $\w$-continuous with respect to the Euclidean metrics in domain and range, with $\w(t) = t^{1/K}$.
\end{corollary}

It would be interesting to know if the same result holds in higher dimensions.

\subsection{Quasiregular dynamics}

If $f:S^n \to S^n$ and all its iterates are $K$-quasiregular for some $K\geq 1$, then $f$ is called uniformly quasiregular. Just as in complex dynamics, the iteration theory of uniformly quasiregular mappings can be studied via normality. More precisely, the Fatou set $F(f)$ is the set of stable behaviour where the family of iterates locally forms a normal family and the Julia set $J(f)$ is the set of chaotic behaviour where the family of iterates does not locally form a normal family. The notion of normal family here is the usual version of Miniowitz with respect to $S^n$, but as we have seen, this is equivalent to our notion of normality, compare with Definition \ref{def:qrnormal}.

For $r>0$, let
\[L(x,f,r) = \sup_{\sigma (y,x)=r} \sigma ( f(y),f(x)).\]
Applying the final statement of Theorem \ref{thm:qr} leads to the following characterization of the Julia set of a uniformly quasiregular mapping.

\begin{theorem}
\label{thm:julia}
Let $f:S^n \to S^n$ be uniformly $K$-quasiregular. Then $x_0\in J(f)$ if and only if there exist sequences $r_k \to 0$ and $m_k \to \infty$ such that
\begin{equation}
\label{eq:julia} 
\lim_{k\to \infty} \frac{L(x_0,f^{m_k},r_k)}{r_k^{\alpha}} = \infty,
\end{equation}
where $\alpha = K^{1/(1-n)}$.
\end{theorem}

\begin{proof}
Suppose $x_0 \in F(f)$. Then there exists $r_0 >0$ such that the family $\{ f^m |_{B_{\sigma}(x_0,r_0)} \}$ is normal. By the final statement of Theorem \ref{thm:qr}, this means that if $r\leq r_0/2$ there exists $C>0$ such that
\[ L(x_0,f^m,r) \leq Cr^{\alpha}\]
for all $m \in \N$. It is then clear that \eqref{eq:julia} cannot hold for any sequences $r_k \to 0$ and $m_k \to \infty$.
On the other hand, suppose $x_0 \in J(f)$ and let $r_k\to 0$. Fix $\delta \in (0,2)$.  By \cite[Lemma 4.2]{Fletcher}, for each $k$ we can find $m_k \in \N$ such that $L(x_0,f^{m_k},r_k) \geq \delta$. Since $r_k \to 0$, we obtain \eqref{eq:julia}.
\end{proof}

\section{Normal quasiregular mappings}

Lehto and Virtanen \cite{LV} introduced normal meromorphic functions $f:\D \to \overline{\C}$ defined by the property that $\{ f\circ A: A \in G \}$ is a normal family, where $G$ is the group of M\"obius automorphisms of $\D$. Normal meromorphic functions were shown to have many properties in common with bounded analytic functions, with the invariance under M\"obius transformations being a key tool. We refer to \cite{Pommerenke} for more details on normal meromorphic functions. 

With this in mind, we make the following analogous definition.

\begin{definition}
\label{def:isonormal}
Let $n\geq 2$.
Let $X \subset S^n$ be a domain, let $(X,d_X)$ be a metric space arising from a conformal metric and let $G$ be a transitive collection of conformal orientation-preserving isometries of $X$ onto itself.
\begin{enumerate}[(i)]
\item We say that a quasiregular mapping $f:X\to S^n$, is a normal quasiregular mapping into $S^n$ if the family $$\mathcal{F} = \{ f\circ A : A \in G \}\subset \mathcal{Q}_K(X,S^n)$$ is a normal family.
\item We say that a quasiregular mapping $f:X\to \R^n$ is a normal quasiregular mapping into $\R^n$ if the family $$\mathcal{F} = \{ f(A(x)) - f(A(x_0)) : A \in G \}\subset \mathcal{Q}_K(X,\R^n)$$ is normal for some $x_0 \in X$.
\end{enumerate}
\end{definition}

In discussing normal quasiregular mappings, there is always implicitly a metric space $(X,d_X)$ and a transitive collection $G$ of isometries of $X$. 

In (ii), the choice of $x_0$ does not matter. To see this, suppose $x_0,x_1$ are distinct points of $X$ and let $\mathcal{F}_0$ and $\mathcal{F}_1$ be the families with respective base-points $x_0$ and $x_1$. Suppose $\mathcal{F}_0$ is normal. As $\mathcal{F}_0(x_0) = \{ 0 \}$, we see that $\mathcal{F}_0$ is relatively compact in $C(X,\R^n)$. By Theorem~\ref{thm:aa}, $\mathcal{F}_0$ is equicontinuous on $X$ and hence so is $\mathcal{F}_1$. Observe that 
\begin{equation}
\label{eq:x0}
f(A(x)) - f(A(x_1)) =[ f(A(x)) - f(A(x_0)) ] - [f(A(x_0)) - f(A(x_1)) ].
\end{equation}
By Theorem~\ref{thm:aa} again, the orbit $\mathcal{F}_0(x)$ is relatively compact in $\R^n$ (i.e.~it is bounded) for any $x\in X$. By \eqref{eq:x0} it follows that $\mathcal{F}_1$ also has bounded orbits. Hence Theorem~\ref{thm:aa} gives that $\mathcal{F}_1$ is relatively compact in $C(X,\R^n)$ and we conclude that $\mathcal{F}_1$ is normal.

It is important to note that the specific choice of $G$ does not matter. Theorem \ref{thm:balls} below shows that normal quasiregular mappings can be characterized via uniform continuity and this latter condition does not involve $G$. Consequently, while $G$ could be chosen to be the full isometry group of $X$, it does not have to be.
Following the conventions laid down in the existing complex analysis literature, we will apply Definition \ref{def:isonormal} where $X$ is either $\B^n$ with the hyperbolic metric and $G$ the full conformal isometry group of $\B^n$, or where $X$ is $\R^n$ with the Euclidean metric and $G$ is the translation group of $\R^n$ (noting that this is a proper subgroup of the full conformal isometry group of $\R^n$).

There is nothing to be gained by taking $X$ to be $S^n$ with the spherical metric, since if $f$ is non-constant then the image of $f(S^n)$ must be all of $S^n$ (and hence case (ii) in Definition \ref{def:isonormal} yields only constant maps) and the group $G$ of orientation-preserving spherical isometries of $S^n$ is compact. Then $\{f\circ A: A\in G \}$ is automatically normal and so every quasiregular mapping $f:S^n \to S^n$ is normal. We again emphasize that the choice of $G$ is irrelevant, but choosing $G$ to be the full isometry group exhibits the point we wish to make here.

Our first main result in this section shows that a normal quasiregular mapping is uniformly continuous with respect to $d_X$ and the spherical distance or the Euclidean distance, respectively. 
This is a generalization of \cite[Theorem 3.2(a)]{MV}.

\begin{theorem}
\label{thm:balls}
Let $X \subset S^n$ be a domain, let $(X,d_X)$ be a metric space arising from a conformal metric and let $G$ be a transitive collection of conformal orientation-preserving isometries of $X$ onto itself. Then 
\begin{enumerate}[(i)]
\item $f:X \to S^n$ is a normal quasiregular mapping if and only if $f$ is uniformly continuous with respect to $d_X$ and $\sigma$;
\item $f:X\to \R^n$ is a normal quasiregular mapping if and only if $f$ is uniformly continuous with respect to $d_X$ and the Euclidean distance.
\end{enumerate}
\end{theorem}

\begin{proof}
We first prove $(i)$.
Suppose that $f$ is normal. Then by Theorem~\ref{thm:qr} and Proposition~\ref{prop:1}, given 
a compact set $E \subset X$, find a constant $L_E>0$ so that for every $x,y\in E$ and $A\in G$, we have
\begin{equation}
\label{eq:sig1} 
\sigma (f(A(x)),f(A(y)) ) \leq L_E d_X(x,y)^{\alpha}.
\end{equation}
Next, fix $x_0 \in X$ and find $r_0>0$ so that $B_X(x_0,r_0)$ is relatively compact in $X$. Let $E$ be the compact set $\overline{B_X(x_0,r_0)}$.
Given $\epsilon >0$, choose 
\[
\delta < \min \left \{ \frac{r_0}{2} , \left ( \frac{ \epsilon }{L_E} \right ) ^{1/\alpha} \right \}.
\]
If $x,y \in X$ with $d_X(x,y) < \delta$, find $A\in G$ such that $A(x_0) = x$. Then $y'=A^{-1}(y) \in B_X(x_0,r_0) \subset E$ and by \eqref{eq:sig1} we have
\begin{align*}
\sigma( f(x) , f(y) ) & = \sigma( f(A(x_0)) , f(A(y')) ) \\
&\leq L_E d_X( x_0 , y')^{\alpha} \\
& = L_E d_X( A^{-1}(x) , A^{-1}(y) )^{\alpha} \\
&= L_E d_X( x,y)^{\alpha}\\
&< L_E \delta^{\alpha} \\
&< \epsilon.
\end{align*}
We conclude that $f$ is uniformly continuous.

For the converse, suppose that $f$ is not normal. Hence $\mathcal{F}$ is not relatively compact in $C(X,S^n)$. By the Arzela-Ascoli Theorem, Theorem \ref{thm:aa}, it follows that $\mathcal{F}$ is not equicontinuous. This means there exists $\epsilon >0$ such that for all $m\in \N$ there exist $x_m,y_m \in X$ and $A_m \in G$ such that $d_X(x_m,y_m) < 1/m$ but $\sigma ( f(A_m(x_m)) , f(A_m(y_m)) ) \geq \epsilon$. Since $A_m$ is an isometry, by setting $u_m = A_m(x_m)$ and $v_m=A_m(y_m)$ we have $d_X(u_m,v_m) < 1/m$ but $\sigma ( f(u_m) , f(v_m) ) \geq \epsilon$ for all $m \in \N$. We conclude that $f$ is not uniformly continuous.

We now turn to (ii). If $f:X\to \R^n$ is normal, then the family $\mathcal{F} = \{ f(A(x)) - f(A(x_0)) : A \in G \}$ is relatively compact in $\mathcal{Q}_K(X,\R^n)$ and we proceed as in part (i) via Theorem \ref{thm:qr} and Proposition \ref{prop:1}.

For the converse, suppose that $f:X\to \R^n$ is not normal and so $\mathcal{F}$ is not relatively compact in $\mathcal{Q}_K(X,\R^n)$. Then by the Arzela-Ascoli Theorem, Theorem \ref{thm:aa}, either $\mathcal{F}$ is not equicontinuous, or $\mathcal{F}$ has an unbounded orbit. If $\mathcal{F}$ is not equicontinuous, then the argument given in (i) shows that $f$ is not uniformly continuous, noting that
\[ |(f(A_m(x_m)) - f(A_m(x_0)) ) - (f(A_m(y_m)) - f(A_m(x_0)) ) |  = | f(A_m(x_m))- f(A_m(y_m)) |.\]

Finally, we have to deal with the case where $\mathcal{F}$ has an unbounded orbit. For a contradiction, suppose that $\mathcal{F}$ has an unbounded orbit and that $f$ is uniformly continuous. Find $x_0,u\in X$ and $A_k \in G$ such that
\begin{equation}
\label{eq:sig2}
|f(A_k(u)) - f(A_k(x_0))| \to \infty
\end{equation}
as $k\to \infty$. Given $\epsilon =1$, find $\delta >0$ so that $d_X(x,y) < \delta$ implies $|f(x) - f(y) | <1$. 
We can find finitely many points $x_1,\ldots, x_m$ in $X$ with $x_m=u$ and $d_X(x_{i-1},x_i) < \delta$, for $i=1,\ldots, m$. If we then set $x_i^k = A_k(x_i)$ for $i=0,\ldots, m$ and $k\in \N$, we have $d_X(x_{i-1}^k,x_i^k) <\delta$ for all $k\in \N$ since $A_k$ is an isometry. By the triangle inequality and uniform continuity, we have for all $k\in \N$ that
\begin{align*} 
|f(A_k(u)) - f(A_k(x_0)) | &\leq \sum  _{i=1}^m | f(A_k(x_i)) - f(A_k(x_{i-1})) | \\
&= \sum_{i=1}^m |f(x_i^k) - f(x_{i-1}^k)| <m. 
\end{align*}
This contradicts \eqref{eq:sig2} and so we conclude that $f$ cannot be uniformly continuous. This completes the proof.
\end{proof}

Another case we will say little about is when the domain and range of $f$ are both $\R^n$ with the Euclidean metric. In the case $n=2$ and we restrict to holomorphic functions, normality reduces to the condition $|f'(z)| \leq C$ for some $C>0$. Consequently $f$ must be linear. 
An application of Theorem \ref{thm:balls}(ii) shows that in the quasiregular case, the uniform continuity of $f$ with respect to Euclidean distances in domain and range yields a growth condition of the form $M(r,f) \leq Cr$ for sufficiently large $r$.

One could consider taking $X$ as a proper sub-domain of $\R^n$ with the quasihyperbolic distance. However, typically the only isometries of the quasihyperbolic distance are M\"obius maps, see \cite{Hasto}, which places strong restrictions on $X$. We will not consider other metric spaces here.

\begin{remark}
In light of Theorem \ref{thm:qr}, for a $K$-quasiregular mapping $f:X\to Y$ one could consider the quantities
\[ Q_f(x) = \limsup_{y\to x} \frac{ d_Y(f(x),f(y) ) }{d_X(x,y)^{\alpha} },\]
and 
\[ Q_f = \sup_{x\in X} Q_f(x),\]
where $\alpha = K^{1/(1-n)}$. It is not hard to show that if $f$ is normal (with $Y$ either $S^n$ or $\R^n$) then $Q_f < \infty$. However the condition that $Q_f< \infty$ is not enough to conclude that $f$ is normal. The reason is that a $K$-quasiregular mapping is also a $K_1$-quasiregular mapping for $K_1>K$. The quantity $Q_f$ depends on $\alpha$ and thus on the choice of $K$.  More specifically, if $g(z) = e^{e^z}$ then since $g$ has infinite order, it cannot be normal as a function $\C \to \C\cup \{\infty \}$ (see Section~\ref{sec:Yosida} below). However, $g$ is $2$-quasiregular and with $\alpha = 1/2$, one can show that $Q_g=0$. This gives a counterexample to \cite[Theorem 1.7]{MV}.
\end{remark}

We next characterize normal quasiregular mappings from $X$ into $S^n$ as globally H\"older continuous. 

\begin{theorem}
\label{thm:globalholder}
Let $X \subset S^n$ be a domain, let $(X,d_X)$ be a metric space arising from a conformal metric and let $G$ be a transitive collection of conformal orientation-preserving isometries of $X$ onto itself.
Let $f:X \to S^n$ be a $K$-quasiregular mapping and let $\alpha = K^{1/(1-n)}$. Then $f$ is normal if and only if $f$ is globally $\alpha$-H\"older, that is, there exists $C>0$ so that
\begin{equation}
\label{eq:gh} 
\sigma ( f(x) , f(y) ) \leq C d_X(x,y)^{\alpha}
\end{equation}
for all $x,y\in X$.
\end{theorem}

This result generalizes \cite[Corollary 13.5]{Vuorinen}, in which $X$ is $\B^n$ with the hyperbolic metric.
Note that if $\beta \leq \alpha$, then this result says that $f$ is $\beta$-H\"older if and only if $f$ is $\alpha$-H\"older. In other words, once we have a H\"older continuous $K$-quasiregular mapping, we can automatically improve the exponent all the way to $K^{1/(1-n)}$.

\begin{proof}
We first assume that $f$ is normal.
Fix $x_0 \in X$ and find $r >0$ so that $B_X(x_0,2r)$ is relatively compact in $X$. Let $E  = \overline{B_X(x_0,r)}$.
By Theorem \ref{thm:qr} and Proposition \ref{prop:1}, find $C>0$ so that
$\sigma(f(A(x)),f(A(y)) ) \leq Cd_X(x,y)^{\alpha}$
for all $x,y \in E$ and all $A\in G$. We can assume that $C\geq \pi/r^{\alpha}$.

Now, if $x,y\in X$ with $d_X(x,y) \geq r$ then \eqref{eq:gh} trivially holds since spherical distances on $S^n$ are bounded above by $\pi $. Otherwise, if $x,y\in X$ with $d_X(x,y) <r$, let $A\in G$ map $x_0$ to $y$. 
Setting $x_1 = A^{-1}(x) \in E$, we have
\begin{align*}
\sigma (f(x),f(y) ) &= \sigma ( f(A(x_1)), f(A(x_0)) ) \\
&\leq Cd_X(x_1,x_0)^{\alpha} \\
&= Cd_X(A(x_1), A(x_0) )^{\alpha} \\
&= Cd_X(x,y)^{\alpha},
\end{align*}
as required.

For the converse, assume that \eqref{eq:gh} holds for all $x,y\in X$. Let $x_0 \in X$, $r>0$ and suppose $x,y\in B_X(x_0,r)$. If $A\in G$, we have
\[ \sigma (f(A(x)) , f(A(y)) ) \leq C d_X(A(x) , A(y) )^{\alpha} = Cd_X(x,y)^{\alpha}.\]
Hence $\{f\circ A : A\in G \}$ is a locally uniformly $\w$-continuous family and hence by the final statement of Theorem \ref{thm:qr}  is normal.
\end{proof}

We next provide a characterization when the range is $\R^n$.

\begin{theorem}
\label{thm:globalholder2}
Let $X \subset S^n$ be a domain, let $(X,d_X)$ be a metric space arising from a conformal metric and let $G$ be a transitive collection of conformal orientation-preserving isometries of $X$ onto itself.
Let $f:X \to \R^n$ be a $K$-quasiregular mapping and let $\alpha = K^{1/(1-n)}$. Then $f$ is normal if and only if there exists $C>0$ so that
\begin{equation}
\label{eq:gh2} 
|f(x) - f(y)| \leq C \max \{ d_X(x,y)^{\alpha} , d_X(x,y) \}
\end{equation}
for all $x,y\in X$.
\end{theorem}

Growth conditions of the form \eqref{eq:gh2} naturally appear in the theory, see for example \cite[Theorem 12.21]{Vuorinen}. This says that while we cannot ignore the H\"older behaviour on small scales, on large scales a normal quasiregular map into $\R^n$ is Lipschitz.

\begin{proof}
We first assume that $f$ is normal. Let $x_0 \in X$ and choose $0<r<1$ so that $B_X(x_0,2r)$ is relatively compact in $X$. Set $E = \overline{B_X(x_0,r)}$. By Corollary \ref{thm:qr3}, find $L>0$ so that
\begin{equation*}
| f(A(x')) - f(A(y')) | \leq Ld_X(x',y')^{\alpha}
\end{equation*}
for all $x',y' \in E$ and all $A\in G$. 
Now suppose $x,y\in X$ with $d_X(x,y) \leq r$. Find $A \in G$ so that $x_0 = A(x)$ and $y'=A(y) \in E$. Then we have
\begin{align}
\label{eq:globalholder2}
|f(x) - f(y)| &= |f(A^{-1}(x_0)) - f(A^{-1}(y')) | \nonumber \\
& \leq L d_X(x_0 , y')^{\alpha} \nonumber\\
&= L d_X(A(x_0) , A(y))^{\alpha} \nonumber\\
&= L d_X(x,y)^{\alpha}.
\end{align}

Now suppose $x,y \in X$ with $d_X(x,y) >r$. Suppose $m\in \N$ such that 
\begin{equation}
\label{eq:gh3}
d_X(x,y) \in ( mr/2 , (m+1)r/2 ]. 
\end{equation}
Since $d_X(x,y) = \inf_{\gamma} \int_{\gamma} \tau_X(u) |du|$, we can find a sequence of points $z_0,z_1,\ldots, z_m$ with $z_0 = x$, $z_{m}=y$, $d_X(z_i,z_{i+1}) = r/2$ for $i=0,\ldots, m-2$ and $d_X(z_{m-1}, z_m) <r$.
Then by \eqref{eq:globalholder2} and \eqref{eq:gh3}, we have
\begin{align*}
|f(x) - f(y)| & \leq \sum_{i=0}^{m-1} |f(z_i) - f(z_{i+1}) | \\
&\leq \sum_{i=0}^{m-1} Ld_X(z_i,z_{i+1})^{\alpha} \\ 
&= \frac{L(m-1)r^{\alpha}}{2^{\alpha}}+ Ld_X(z_{m-1} , z_{m})^{\alpha}\\
&\leq Lmr^{\alpha} \\
&\leq 2Lr^{\alpha-1} d_X(x,y).
\end{align*}
It follows that \eqref{eq:gh2} holds globally with $C = 2Lr^{\alpha-1} $.

For the converse, suppose that \eqref{eq:gh2} holds for all $x,y\in X$. Let $x_1 \in X$, choose $0<r<1/2$ and suppose $x,y\in B_X(x_1,r)$. If $A\in G$, we have
\[ |f(A(x)) - f(A(y)) | \leq C d_X(A(x) , A(y))^{\alpha} = Cd_X(x,y)^{\alpha}.\]
Thus $\{f(A(x)) - f(A(x_0)) : A\in G \}$ is a locally uniformly $\w$-continuous family and hence by Corollary \ref{thm:qr3}, $f$ is normal.
\end{proof}

\section{Classes of normal quasiregular mappings}

In this section, we investigate three classes of normal quasiregular mappings.

\subsection{Bloch mappings}

A Bloch function is a holomorphic function $f:\D \to \C$ satisfying
\[ \sup _{z\in \D } (1-|z|^2)|f'(z)| < \infty.\]
Bloch functions have been much studied, see for example \cite{Pom2}.
Bloch functions can be characterised as holomorphic functions which are Lipschitz with respect to the hyperbolic distance in the domain and the Euclidean distance in the range, or equivalently, as functions where the family $\mathcal{F} = \{ f(A(x)) - f(A(0)) : A\in G\}$ is normal, where $G$ is the group of M\"obius automorphisms of the disk, see \cite{Po}.
A profitable viewpoint for the study of Bloch functions has been to put a norm on the space of Bloch functions via
\[ ||f|| _{\mathcal{B}} = \sup_{z\in \D} (1-|z|^2) |f'(z)| + |f(0)|.\]
The term $|f(0)|$ here is to be able to distinguish between constant functions and without it we just have a semi-norm.

In our quasiregular setting, a Bloch mapping is a special case of Definition \ref{def:isonormal}(ii) when $X$ is $\B^n$ equipped with the hyperbolic distance, and we take $G$ to be the full conformal isometry group of $\B^n$. We will denote this group by $G$ throughout this section.

\begin{definition}
\label{def:bloch}
Let $n\geq 2$.
A quasiregular mapping $f:\B^n \to \R^n$ is called a {\it Bloch mapping} if the family $\mathcal{F} = \{ f(A(x)) - f(A(0)) : A\in G\}$ is normal.
\end{definition}

It is evident from the discussion above that if $f$ is a holomorphic function on $\D$, then it is a Bloch function if and only if it is a Bloch mapping in the sense of Definition \ref{def:bloch}.
Bloch mappings exist in all dimensions as the following example shows.

\begin{example}
Suppose $f:\B^n \to \R^n$ is bounded, say $f(\B^n) \subset B(0,R)$ for some $R>0$. Then for any $A\in G$, the image $f(A(\B^n)) - f(A(0))$ must be contained in $B(0,2R)$. By Theorem \ref{thm:montel}, we can conclude that the family $\{ f(A(x)) - f(A(0)) : A \in G\}$ is normal. Hence every bounded quasiregular map $f:\B^n \to \R^n$ is a Bloch mapping.
\end{example}

We will give a characterization of Bloch mappings via the Bloch radius. If $f:\B^n \to \R^n$ is $K$-quasiregular then following, for example, \cite{Eremenko}, we define the \emph{Bloch radius} $\mathcal{B}_f$ to be the supremum of all radii $r>0$ such that there exists a domain $U \subset \B^n$ with the property that $f|_U$ is bijective and $f(U)$ is a ball of Euclidean radius $r$. We remark that Rajala \cite[Theorem 1.5]{Rajala} has proved a quantitative version of a result of Eremenko \cite{Eremenko}, namely that if $f$ is as above, then 
\[ \mathcal{B}_f \geq C \operatorname{diam} f(B(0,1/2) ),\]
where the diameter here is in the Euclidean distance, and $C>0$ depends only on $n$ and $K$.

\begin{example}
Let $f: \B^3 \to \mathbb{H}^3_+ = \{ (x_1,x_2, x_3) \in \R^3 : x_3 >0 \}$ be a surjective M\"obius map. Then clearly the Bloch radius of $f$ is infinite.
\end{example}

\begin{theorem}
\label{thm:blocheq}
Let $n\geq 2$ and $K\geq 1$. Let $f:\B^n \to \R^n$ be a $K$-quasiregular mapping. Then $f$ is a Bloch mapping if and only if the Bloch radius $\mathcal{B}_f$ is finite.
\end{theorem}

\begin{proof}
First suppose that $\mathcal{B}_f$ is finite. Clearly the Bloch radius is unchanged if we replace $f$ by $f\circ A$ for $A\in G$. Then by \cite[Theorem 1(iii)]{Eremenko}, the family $\{ f\circ A : A\in G\}$ is equicontinuous with respect to the Euclidean distance in domain and range. In particular, if $E\subset \B^n$ is compact then for any $\epsilon >0$, there exists $\delta >0$ so that if $x,y\in E$ with $|x-y| <\delta$, then $|f(A(x)) - f(A(y))|< \epsilon$. Since $\rho(x,y) \geq |x-y|$ and $\rho(A(x),A(y)) = \rho(x,y)$ for all $A\in G$, we conclude that $f$ is uniformly continuous with respect to the hyperbolic and Euclidean distances. Theorem \ref{thm:balls}(ii) implies that $f$ is normal and hence a Bloch mapping.

On the other hand, suppose that $f$ is a Bloch mapping and the Bloch radius is infinite. Then there exist sequences $(x_m)$ in $\B^n$, $r_m\to \infty$ and neighbourhoods $U_m$ of $x_m$ so that $f$ is bijective on $U_m$ and $f(U_m) = B(f(x_m), r_m)$. Let $A_m \in G$ be a M\"obius map sending $0$ to $x_m$ and define 
\[ g_m(x) = f(A_m(x)) - f(A_m(0)).\]
Further, for $x\in \B^n$, set
\[ h_m(x) = g_m^{-1}(r_m x),\]
where we take the branch of the inverse corresponding to $(f|_{U_m})^{-1}$. Since each $h_m$ is a $K$-quasiconformal map from $\B^n $ into $\B^n$, Theorem \ref{thm:montel} implies $\{h_m : m \in \N\}$ is normal and, by passing to a subsequence if necessary, we may assume that $h_m \to h_0$ locally uniformly. Since $h_m(0)=0$ for all $m$, we conclude that $h_0(0)=0$.

Fix $x_0 \neq 0$ in $\R^n$ and, for $m$ large enough so that $r_m > |x_0|$, find $w_m \in A_m^{-1}(U_m)$ so that $g_m(w_m) = x_0$. Moreover, let $u_m \in \B^n$ be such that $h_m(u_m) = w_m$. Since $u_m = x_0 / r_m$ by construction, we have $u_m \to 0$. Since $h_m\to h_0$ uniformly on compact subsets of $\B^n$ and $u_m\to 0$, we find that $h_m(u_m) \to h_0(0)=0$, that is, $w_m \to 0$.

We will now aim for a contradiction. To that end, choose $M\in\N$ large enough so that $w_m \in E= \overline{B_{\rho}(0,1)}$ for $m\geq M$. Since $E$ is compact and $\{g_m : m\in \N \}$ is a finitely normal family of $K$-quasiregular maps, Corollary \ref{thm:qr3} implies the existence of a constant $C>0$ so that
\[ |g_m(x) | \leq C \rho (x,0)^{\alpha}\]
for all $x\in E$. Applying this to $w_m$, and recalling that $g_m(w_m) = x_0$, we obtain
\[ \rho(w_m,0) \geq \left ( \frac{ |x_0|}{C} \right )^{1/\alpha}\]
for $m\geq M$. This contradicts the earlier established fact that $w_m\to 0$.
\end{proof}

An analogue for the Bloch norm in the quasiregular setting is the following.

\begin{definition}
\label{def:blochnorm}
Suppose $f:\B^n \to \R^n$ is a quasiregular Bloch mapping. Then we define
\[ R_f := |f(0)| + \sup_{x\in \B^n} \operatorname{diam} f(B_{\rho}(x,1) ).\]
\end{definition}

Clearly, we cannot use the version of the Bloch norm for holomorphic functions in $\D$ since quasiregular mappings are only differentiable almost everywhere. Therefore, we can look at the largest Euclidean diameter of images of hyperbolic balls of a fixed radius instead. The choice of radius $1$ here is arbitrary. In fact, if $0<s<t$, then any hyperbolic ball in $\B^n$ of radius $t$ can be covered by $C$ balls of radius $s$, where $C$ depends only on $n, s$ and $t$. In particular, whether $R_f$ is finite or infinite is  not changed if $1$ is replaced by any other radius.

\begin{theorem}
\label{lem:blochnorm}
A quasiregular mapping $f:\B^n\to\R^n$ is a Bloch mapping if and only if $R_f < \infty$.
\end{theorem}

\begin{proof}
Suppose $f$ is a Bloch mapping. Then since $E = \overline{B_{\rho}(0,1)}$ is compact, by Corollary~\ref{thm:qr3} there exists $C>0$ such that
\begin{equation}
\label{eq:blochnorm}
|f(A(x)) - f(A(y)) | \leq C \rho (x,y)^{\alpha}
\end{equation}
for all $x,y \in E$ and all $A\in G$. Now, if $x_0 \in \B^n$, let $A_0 \in G$ map $0$ to $x_0$. Suppose $u \in B_{\rho}(x_0,1)$ and set $x=A_0^{-1}(u) \in E$. 
Then applying \eqref{eq:blochnorm} with $y=0$ and $A=A_0$, and using the fact that $A_0$ is an isometry with respect to $\rho$, we see that
\begin{align*} 
|f(u)-f(x_0)|  &= |f(A_0(x)) - f(A_0(0)) | \\
& \leq C \rho(x , 0)^{\alpha} \\
&= C\rho( A_0^{-1}(u) , A_0^{-1}(x_0) ) \\
&=  C \rho(u, x_0)^{\alpha} \\
&\leq C
\end{align*}
for all  $u \in B_{\rho}(x_0,1)$. Since $x_0$ was arbitrary, it follows that $R_f<\infty$.

For the converse, suppose $f$ is not a Bloch mapping and so by Theorem \ref{thm:blocheq} the Bloch radius is infinite. As in the second half of the proof of Theorem \ref{thm:blocheq}, this means we can find sequences $(x_m)$ in $\B^n$, $r_m\to \infty$ and neighbourhoods $U_m$ of $x_m$ so that $f$ is bijective on $U_m$ and $f(U_m) = B(f(x_m), r_m)$. Moreover, let $A_m$ be a M\"obius mapping sending $0$ to $x_m$, let $g_m(x) = f(A_m(x)) - f(A_m(0))$ and, for $m$ large enough, let $h_m:\B^n\to\B^n$ be defined by $h_m = g_m^{-1}(r_mx)$.

Let $0<r<1$ and set $E = \overline{B(0,r)}$. Theorem \ref{thm:montel} implies $\{h_m : m \in \N\}$ is finitely normal.  Corollary~\ref{thm:qr3} and Proposition \ref{prop:1} imply that $\{h_m :m\in \N\}$ is uniformly $\w$-continuous on $E$ with respect to Euclidean distance. In particular, since $h_m(0) = 0$ for all $m$, there exists $L>0$ so that
\begin{equation*}
|h_m(x)| \leq L |x|^{\alpha}
\end{equation*}
for all $x\in E$. Find $0<s\leq r$ so that $Ls^{\alpha }<\frac{e-1}{e+1}$. Then
\begin{equation}
h_m(\overline{B(0,s)}) \subset \overline{B(0,Ls^{\alpha})} \subset \overline{B_{\rho}(0,1)}
\end{equation} 
for all $m$. This implies that
\[ g_m( \overline{ B_{\rho}(0,1 ) } )  \supset \overline{B(0,r_ms)} .\]
Since $r_m\to \infty$, we conclude that 
\[ \sup_{x\in \B^n} \diam f( B_{\rho}(x,1 )) = \infty\]
and hence $R_f = \infty$.
\end{proof}

Let us remark that unlike the usual Bloch space of holomorphic Bloch functions, the set of quasiregular Bloch mappings does not form a normed vector space as it is not closed under addition. On the other hand, it is easy to see that $R_f \geq  0$ with equality if and only if $f$ is the constant map $0$.

Our next result shows that we can use $R_f$ to quantify the growth of Bloch mappings.

\begin{theorem}
\label{thmLblochgrowth}
Let $f:\B^n \to \R^n$ be a quasiregular Bloch mapping. Then
\[ M(r,f)  \leq R_f \max  \left \{ 1 , \log \frac{1+r}{1-r} \right \}.\]
\end{theorem}

\begin{proof}
Let $r>0$ and suppose $|x_r| = r$ with $M(r,f) = |f(x_r)|$. Consider a chain of closed hyperbolic balls $U_1,\ldots, U_m$ each of hyperbolic radius $1$, so that $U_{i} \cap U_{i+1}$ consists of one point, and with $0 \in \partial U_1$ and $x_r \in U_m$. This can be achieved with $m \leq \max \{ 1 , \rho(0,r) \}$ along a radial line segment. If $C = \sup_x \diam f(B_{\rho}(x,1) )$, then
\[ |f(x_r)| \leq |f(0)| + Cm.\]
Since $\rho(0,r) = \log \frac{1+r}{1-r}$, we conclude that
\[ M(r,f) \leq |f(0)| + C\max \left \{ 1 , \log \left ( \frac{1+r}{1-r} \right ) \right \}\leq R_f \max  \left \{ 1 , \log \frac{1+r}{1-r} \right \}\]
as required.
\end{proof}

Recalling that the order of growth of a map $f:\B^n \to \R^n$ is given by 
\[ \mu_f = \limsup_{r\to 1} (n-1) \frac{\log \log M(r,f)}{\log 1/(1-r)},\]
we see that quasiregular Bloch mappings have zero order of growth.

In the literature, the little Bloch space $\mathcal{B}_0$ has been studied for holomorphic functions. There is an analogous definition in this setting.

\begin{definition}
\label{def:littlebloch}
We say that a quasiregular Bloch mapping is a \emph{little Bloch mapping} if
\[ \lim_{|x| \to 1} \operatorname{diam} f(B_{\rho}(x,1) ) = 0.\]
\end{definition}

For example, linear maps on $\B^n$ are clearly little Bloch mappings. There do exist Bloch mappings which are not little Bloch mappings, as is shown by the following natural generalization of the well-known example of the holomorphic Bloch map $\log(1-z)$.

\begin{example}
For simplicity, take $n=3$ and consider a standard Zorich map construction~$\mathcal{Z}$ such as that in \cite[Section 3.1]{FN} from the beam $\Omega = \{(a,b,c) : |a| < \pi/2, |b| < \pi/2 \}$ onto the upper half space $\{(a,b,c) : c> 0 \}$. The key point is that $\mathcal{Z}$ maps slices of the beam at height $c=r$ onto hemispheres in the upper half space of radius $e^r$ centred at the origin. With the given domain and range, $\mathcal{Z}$ is quasiconformal. Define $f:\B^3 \to \R^3$ via the map $f(x)= \mathcal{Z}^{-1}(x+(0,0,1))$.

Since $f(\B^3) \subset \Omega$, it is clear that the Bloch radius $\mathcal{B}_f$ of $f$ is finite. By Theorem \ref{thm:blocheq}, we see that $f$ is indeed a Bloch mapping.

We show that $f$ is not a little Bloch mapping as follows. For $0<r<1$, take 
\[ x=(0,0,-r) \quad \mbox{ and } \quad y=\left(0,0, -\frac{3r+1}{r+3}\right). \]
Then $\rho(x,y) = \log 2$, so $y\in B_\rho(x,1)$. Since the distance from the origin of $x+(0,0,1)$ is $1-r$, the third component of $f(x)$ is $\log(1-r)$. Similarly, $|y+(0,0,1)|= \frac{2-2r}{r+3}$ and so the third component of $f(y)$ is $\log\frac{2-2r}{r+3}$. Hence
\[ |f(x) - f(y)| \ge \left|\log(1-r) - \log\frac{2-2r}{r+3}\right| = \log\frac{r+3}{2}. \]
Therefore $\liminf_{r\to 1} \diam f(B_{\rho} (x,1) ) \ge\log 2>0$ and thus $f$ is not a little Bloch mapping.
\end{example}

As a final remark in this section, holomorphic Bloch functions are studied with respect to their boundary behaviour. It would be interesting to study quasiregular Bloch mappings from this viewpoint too. We refer to \cite{LiVi} for the boundary behaviour of bounded quasiregular mappings in $\B^n$. It would be natural to ask whether the results there can be generalized to Bloch mappings.

\subsection{Normal quasiregular mappings $\B^n \to S^n$}

In this section, we consider normal quasiregular mappings $f:\B^n \to S^n$. This class has already been the subject of study, sometimes under the name normal quasimeromorphic mappings, and we refer to \cite[Chapter 13]{Vuorinen}. 
There, a quasiregular map $f:\B^n \to S^n$ is called normal if it is uniformly continuous with respect to the hyperbolic and spherical distances in the domain and range respectively. By Theorem \ref{thm:balls}, this definition of a normal mapping agrees with Definition \ref{def:isonormal}(i) with $X = \B^n$ equipped with the hyperbolic distance and $G$ the full group of conformal isometries on $\B^n$.

\begin{definition}
A quasiregular map $f:\B^n \to S^n$ is called a \emph{normal quasiregular mapping from $\B^n$ into $S^n$} if the family $\{ f\circ A : A \in G \}$ is normal.
\end{definition}

Here, we prove a result on the order of growth of a normal quasiregular mapping. We refer to the terminology established in section 3.1, see also \cite[Chapter IV]{Rickman}.
Given $f:\B^n \to S^n$ quasiregular, we modify \eqref{eq:afr} so that $A_{f,\sigma}^{\rho}(x_0,r)$ is the normalized volume of the image of a hyperbolic ball $f(\overline{B_{\rho}(x_0,r)})$ with multiplicity taken into account (see \cite[p.80]{Rickman}), that is,
\[ A_{f,\sigma}^{\rho}(x_0,r) = \frac{1}{\omega_n} \int_{ \overline{B_{\rho}(x_0,r)}} \frac{ J_f(x)}{(1+|f(x)|^2)^n} \: dm(x).\]

\begin{lemma}
\label{lem:arho}
If $f:\B^n \to S^n$ is a normal quasiregular mapping, then there exist constants $\delta,C>0$ so that $A_{f,\sigma}^{\rho}(x,\delta ) \leq C$ for every $x\in \B^n$.
\end{lemma}

\begin{proof}
Since $f:\B^n \to S^n$ is normal, $f$ is uniformly continuous with respect to the hyperbolic and spherical distances respectively by Theorem \ref{thm:balls}. Hence there exists $\delta >0$ such that
\[ \sup_{x\in \B^n} \diam_{\sigma} f(B_{\rho}(x,2\delta) ) <\frac{1}{2}.\]
Now let $x\in \B^n$ and find $M\in G$ such that $M(0)=x$. Let $r>0$ and $\theta >1$ be chosen so that $B_{\rho}(0,\delta) = B(0,r)$ and $B_{\rho}(0,2\delta) = B(0,\theta r)$. Then
\[ \diam_{\sigma} (f\circ M)(B(0,\theta r)) = \diam _{\sigma} f(B_{\rho}(x,2\delta)) < \frac{1}{2}.\]
Pick an $(n-1)$-sphere of spherical radius $1/2$ such that $Y \cap (f\circ M)(B(0,\theta r)) = \emptyset$. Then $\nu_{f\circ M} (\theta r, Y) = 0$. By  \cite[Lemma IV.1.7]{Rickman} applied to $f\circ M$, we have
\[ A_{f,\sigma}^{\rho}(x,\delta) = A_{f\circ M,\sigma}^{\rho}(0,\delta) =
A_{f\circ M, \sigma}( r ) \leq \frac{ K_I Q (\log 2)^{n-1} }{(\log \theta )^{n-1} },\]
where $Q>0$ depends only on $n$ and we recall that $K_I$ is the inner dilatation of $f$. Since this holds for all $x\in \B^n$, the lemma follows.
\end{proof}

\begin{theorem}
\label{thm:normalgrowth}
Let $n\geq 2$ and $f:\B^n \to S^n$ be a normal quasiregular mapping. Then the order of growth of $f$ is at most $n-1$.
\end{theorem}

\begin{proof}
Let $r\in (0,1)$ and $t = \rho(0,r)$. Then $A_{f,\sigma}^{\rho}(0,t) = A_{f,\sigma}(0,r)$. 
By \cite[p.79]{Ratcliffe}, the hyperbolic volume of $B_{\rho}(x,t)$ is
\[ \operatorname{Vol}_{\rho}(B_{\rho}(x,t)) = \omega_{n-1} \int_0^t \sinh ^{n-1} u \: du \sim C_n e^{(n-1)t}\]
as $t\to \infty$, where $C_n$ depends only on $n$. Here $f(t) \sim g(t)$ as $t\to \infty$ means that $f(t)/g(t) \to 1$ as $t\to \infty$.

Hence, given $\delta>0$, the number of balls of hyperbolic radius $\delta$ needed to cover $B_{\rho}(x,t)$ is bounded above by $D_ne^{(n-1)t}$ for large $t$, where $D_n$ depends only on $n$ and $\delta$.
Then, by Lemma~\ref{lem:arho}, and using the constants $\delta$ and $C$ from Lemma \ref{lem:arho}, we have
\[ A_{f,\sigma}(0,r) = A_{f,\sigma}^{\rho}(0,t) \leq CD_n e^{(n-1)t} = CD_n \left ( \frac{1+r}{1-r} \right )^{n-1}.\]
Hence
\[ \mu_f = \limsup_{r\to 1} \frac{ \log A_{f,\sigma}(0,r) }{\log \frac{1}{1-r} } \leq n-1. \qedhere\]
\end{proof}

We observe that \cite[Theorem 13.10]{Vuorinen} gives an order of growth estimate depending on $n$ and $K$ (the constant $\gamma$ in that result) related to normal quasiregular mappings $\B^n\to S^n$ without poles. Theorem \ref{thm:normalgrowth} can be viewed as a variation of this result with the dependence on $K$ dropped, and poles of $f$ permitted.

\subsection{Yosida mappings}\label{sec:Yosida} 

A Yosida function in the plane is a meromorphic function $f:\C \to S^2$ such that the family $\mathcal{F} = \{ f(z+w) : w\in \C \}$ is normal. We make an analogous definition  by using Definition \ref{def:isonormal}(i) with $X = \R^n$ equipped with the Euclidean distance, and $G$ the translation group of $\R^n$.

\begin{definition}
\label{def:yosida}
Let $n\geq 2$. 
A quasiregular mapping $f:\R^n \to S^n$ is called a {\it Yosida mapping} if the family $\mathcal{F} = \{ f(x+a) : a\in \R^n\}$ is normal.
\end{definition}

Using an approach similar to that behind Theorem~\ref{thm:normalgrowth}, we can recover a bound on the order of growth of Yosida mappings contained in \cite[Corollary~2.2]{BH}.

\begin{theorem}
\label{thm:yosidagrowth}
Let $f:\R^n \to S^n$ be a Yosida quasiregular mapping. Then the order of growth of $f$ is at most $n$.
\end{theorem}

\begin{proof}
Since $f$ is a Yosida mapping, Theorem \ref{thm:balls} implies that $f$ is uniformly continuous with respect to the Euclidean and spherical distances. Hence there exists $\delta >0$ so that
\[ \sup _{x\in \R^n} \diam _{\sigma} f(B(x,2\delta)) < \frac12.\]
The same argument as in Lemma \ref{lem:arho} (except without needing to amend $A_{f,\sigma}$ to include the hyperbolic distance) implies the existence of $C>0$ such that 
\[ A_{f,\sigma}(x,\delta) \leq C \]
for all $x\in \R^n$. Since there exists a constant $D = D(n,\delta)$ so that we may cover a ball of large radius $r$ by at most $D r^n$ balls of radius $\delta$, we have
\[ A_{f,\sigma} (0,r) \leq CDr^n.\]
Hence
\[ \mu_f = \limsup_{r\to \infty} \frac{\log A_{f,\sigma}(0,r) }{\log r} \leq n. \qedhere\]
\end{proof}

If $f:\C \to \C$ is holomorphic and also a Yosida function, then the order is at most $1$, see \cite[Theorem 3]{CH}. We believe the analogous result is true for entire quasiregular mappings and hence we make the following conjecture.

\begin{conjecture}
Let $f:\R^n \to \R^n$ be an entire quasiregular Yosida mapping. Then the order of $f$ is at most $n-1$.
\end{conjecture}

Examples of Yosida mappings include Zorich mappings and certain other periodic quasiregular mappings, see for example \cite[Sections~7 and~8]{MS}. The standard Zorich mappings on $\R^n$ have order $n-1$ (and more generally see \cite[Corollary~8.11]{MS}), which would imply the sharpness of the above conjecture.

Note that Theorem \ref{thm:globalholder} states in this setting that a $K$-quasiregular mapping $f:\R^n \to S^n$ is Yosida if and only if it is $\alpha$-H\"older continuous, that is, there exists $C>0$ so that
\begin{equation}
\label{eq:yos} 
\sigma(f(x),f(y)) \leq C|x-y|^{\alpha},
\end{equation}
for all $x,y\in\R^n$,
recalling that $\alpha = K^{1/(1-n)}$.
We will prove the following characterization of Yosida mappings, generalizing \cite[Theorem 1]{Minda} for Yosida functions in the plane and closely following the proof of Miniowitz's version of Zalcman's Lemma \cite[Lemma 1]{Min}.

\begin{theorem}
\label{thm:yosida}
Let $n\geq 2$, $K>1$ and $f:\R^n \to S^n$ be $K$-quasiregular. Then $f$ is not a Yosida mapping if and only if there exist sequences $x_m$ in $\R^n$ and $\rho_m >0$ with $\rho_m \to 0^+$ so that $f(x_m + \rho_mw) \to g(w)$ locally uniformly on $\R^n$, where $g$ is a non-constant Yosida mapping.
\end{theorem}

\begin{proof}
First, suppose that there are sequences $x_m,\rho_m$ as in the hypotheses and a non-constant Yosida mapping $g$ so that $f(x_m + \rho_mw) \to g(w)$ locally uniformly on $\R^n$. Further, assume that $f$ is a Yosida mapping. Then by \eqref{eq:yos}, for any $w\in \R^n$, we have
\[ \operatorname{diam}_{\sigma}( g( \overline{ B(w,1) } ) ) = \lim_{m\to \infty} \operatorname{diam}_{\sigma} ( f ( \overline{ B(x_m + \rho_mw, \rho_m ) } ) ) \leq \lim_{m\to \infty} C (2 \rho_m)^{\alpha}.\]
It follows that $\operatorname{diam}_{\sigma} (  g( \overline{ B(w,1) } ) )  = 0$. This implies that $g$ is constant, yielding a contradiction.

For the converse, suppose now that $f:\R^n \to S^n$ is not a Yosida mapping and so ${\mathcal{F} = \{f(x+a) : a\in \R^n \}}$ is not a normal family. Then $\mathcal{F}$ is not normal on $B(0,r^*)$ for some $r^* >0$. By Theorem \ref{thm:qr}, there exists $f_m \in \mathcal{F}$ and $x_m^*,y_m^* \in B(0,r^*)$ with
\[ \lim_{m\to \infty} \frac{ \sigma (f_m(x_m^*) , f_m(y_m^*) ) }{|x_m^* - y_m^* |^{\alpha} } = \infty.\]
Fix $r>r^*$ and set
\[ M_m = \sup_{|x|,|y| \leq r } \left ( 1 - \frac{|x|^2}{r^2} \right) \frac{ \sigma( f_m(x) , f_m(y) )}{|x-y|^{\alpha} }.\]
Clearly $M_m \to \infty$. Now choose $x_m,y_m \in B(0,r)$ with
\[ \left ( 1 - \frac{|x_m|^2}{r^2} \right ) \frac{ \sigma (f_m(x_m) , f_m(y_m) ) }{|x_m - y_m|^{\alpha} } \geq \frac{M_m}{2}\]
and set 
\begin{equation}
\label{eq:yosida1} 
g_m(w) = f_m( x_m + \rho_mw), \quad \mbox{ where } \quad \rho_m^{\alpha} = \frac{ |x_m - y_m| ^{\alpha} }{\sigma ( f_m(x_m) , f_m(y_m) ) }.
\end{equation}
Since 
\[ \rho_m^{\alpha} \leq \left ( \frac{2}{M_m} \right ) \left ( 1 - \frac{|x_m|^2}{r^2} \right ),\]
we have $\rho_m \to 0$ as $m \to \infty$. Now, $|x_m + \rho_m w| \leq r$ if $|w|\leq R_m$, where $R_m = (r-|x_m|)/\rho_m$. As
\[ \frac{1}{R_m} = \frac{\rho_m}{ r - |x_m| } \leq \frac{ 2(r+|x_m| )\rho_m^{1-\alpha} }{r^2 M_m} \leq \frac{4\rho_m^{1-\alpha} }{rM_m},\]
and $\alpha <1$, it follows that $R_m \to \infty$ as $m\to \infty$. We next show that $(g_m)$ forms a normal family on $\R^n$. 

Let $R>0$ and let $w_1,w_2 \in B(0,R)$. Choose $m$ large enough that $R<R_m$ and so $|x_m + \rho_m w_i | <r$ for $i=1,2$. 
Then by the definitions of $g_m$ and $M_m$, we have
\begin{align} \nonumber
\frac{ \sigma (g_m(w_1) , g_m(w_2) ) }{|w_1 - w_2|^{\alpha}} &= \frac{ \sigma ( f_m( x_m + \rho_m w_1 ) , f_m( x_m + \rho_m w_2) )}{|w_1 - w_2 |^{\alpha} } \\
&=  \nonumber\frac{ \rho_m^{\alpha} \sigma ( f_m ( x_m + \rho_m w_1) , f_m ( x_m + \rho_m w_2 ) ) }{| ( x_m + \rho_m w_1) - (x_m + \rho_mw_2) |^{\alpha} }\\
& \nonumber\leq \frac{ \rho_m^{\alpha} M_m}{  1 - \frac{ |x_m + \rho_m w_1 | ^2}{r^2}  } \\
& \nonumber \leq \frac{ \left ( \frac{2}{M_m} \right ) \left ( 1 - \frac{ |x_m|^2}{r^2} \right ) M_m}{1 - \frac{ |x_m + \rho_mw_1|^2}{r^2} } \\
&= \nonumber \frac{2 ( r^2 - |x_m|^2) }{r^2 - |x_m + \rho_mw_1|^2} \\
& \label{eq:yosida3} \leq 2 \left ( \frac{r+|x_m|}{ r+|x_m| +R\rho_m } \right ) \left ( \frac{ r-|x_m|}{r-|x_m| - R\rho_m } \right ).
\end{align}
The first ratio here is bounded above by $1$ and the second converges to $1$ as $m\to \infty$ for $R$ fixed.
Consequently, by Theorem \ref{thm:qr}, $(g_m)$ forms a normal family. 

We need to show that a limit $g$ of a convergent subsequence of $(g_m)$ is non-constant. To that end, find $\xi_m$ with $x_m+\rho_m\xi_m = y_m$. Then by \eqref{eq:yosida1}
\[ |\xi_m| = \frac{|y_m - x_m|}{\rho_m} = \frac{ \left|y_m - x_m\right| \sigma( f_m(x_m) , f_m(y_m) )^{1/\alpha} }{|x_m-y_m|} \leq \pi^{1/\alpha}.\]
Therefore $(\xi_m)$ contains a convergent subsequence which we pass to and relabel. Note that
\begin{equation}
\label{eq:yosida2}
 \sigma(g_m(0) , g_m(\xi_m) )  = \sigma(f_m(x_m) ,  f_m(y_m) ) = \frac{|x_m-y_m|^{\alpha}}{\rho_m^{\alpha} } = |\xi_m|^{\alpha}.
\end{equation}
We may assume that $\xi_m\to0$, since otherwise any limit function $g$ will take distinct values at $0$ and $\lim \xi_m$. Now let $h_m = U_m \circ g_m$, where $U_m$ is a spherical isometry sending $g_m(0)$ to~$0$. By the equicontinuity of $(g_m)$ we can find $a>0$ such that the $h_m$ are uniformly  bounded $K$-quasiregular mappings on $B(0,a)$ for all large $m$, and hence form a normal family.
 We have
\begin{equation}
\label{eq:yosida4}
 \sigma(g_m(0) , g_m(\xi_m) )  = \sigma(h_m(0) , h_m(\xi_m) ) \le 2 |h_m(\xi_m)|,
\end{equation}
recalling the fact that $\sigma(x,y)\le 2|x-y|$.
For large $m$, we now apply Theorem \ref{thm:ric} with ${U = B(0,a)}$, ${f= h_m}$, ${E = \{ 0 \}}$, ${y = \xi_m \in B(0,a)}$, ${\delta = d(0,\partial U) = a}$ and ${r= \operatorname{diam}(h_m(B(0,a)))}$ to get
\[ |h_m(\xi_m)| \leq \frac{ \lambda(n) |\xi_m|^{\alpha} \operatorname{diam}(h_m(B(0,a))) }{a^{\alpha} } .\]
Recall that $\lambda(n)$ is a constant depending only on $n$. Combining this with \eqref{eq:yosida2}, \eqref{eq:yosida4} we conclude that
\[ \operatorname{diam}(h_m(B(0,a))) \geq \frac{a^{\alpha}}{2\lambda(n)}.\]
Consequently, once we pass to a convergent subsequence of $(h_m)$, we see that any limit must be non-constant. Since $h_m = U_m \circ g_m$, the same is therefore true for any limit of a subsequence of $(g_m)$.

Finally, we have to show that any such limit $g$ is a Yosida mapping. For any $R>0$ and any $w_1,w_2 \in B(0,R)$, \eqref{eq:yosida3} implies that
\[ \sigma ( g(w_1) , g(w_2) ) \leq 2|w_1-w_2|^{\alpha}.\]
Consequently this is also true for any $w_1,w_2 \in \R^n$. This shows that $g:\R^n \to S^n$ is globally H\"older continuous and hence by Theorem \ref{thm:globalholder}, $g$ is normal. In other words, $g$ is a Yosida mapping, as required.
\end{proof}

\end{document}